\documentclass[10pt,a4paper,leqno]{amsart}
\usepackage{amssymb,xspace}
\usepackage{amstext}
\usepackage[mathscr]{eucal}
\theoremstyle{plain}
\usepackage{amsbsy,amssymb,amsfonts,latexsym,eucal,amscd}
\usepackage[dvips]{graphicx}
\usepackage{epsfig}
\usepackage[all]{xy}
\usepackage{pstricks}
\usepackage{mathrsfs}
\newcommand{\tens}[1][]{\mathbin{\otimes_{\raise1.5ex\hbox to-.1em{}{#1}}}}
\newcommand{\lltens}[1][]{{\mathop{\tens}\limits^{\rm \mathbb{L}}}_{#1}}
\newcommand{\lltensl}[1][]{{\mathop{\tens}\limits^{\rm \mathbb{L}, \, \ell}}_{#1}}
\newcommand{\lltensr}[1][]{{\mathop{\tens}\limits^{\rm \mathbb{L},\, r}}_{#1}}

\newcommand{\ds }{\ensuremath{\displaystyle}}

\newcommand{\st }{\ensuremath{\scriptstyle}}

\newcommand{\C }{\ensuremath{\mathbb C}}

\newcommand{\Z }{\ensuremath{\mathbb Z}}
\newcommand{\N }{\ensuremath{\mathbb N}}

\DeclareMathOperator{\hoo}{\mathcal{H}\mathnormal{om}}

\DeclareMathOperator{\pr}{pr}

\newcommand{\aaa }{\ensuremath{\mathcal{A}}}
\newcommand{\bb }{\ensuremath{\mathcal{B}}}

\newcommand{\oo }{\ensuremath{\mathcal{O}}}
\newcommand{\hh }{\ensuremath{\mathcal{H}}}
\newcommand{\ff }{\ensuremath{\mathcal{F}}}
\newcommand{\kk }{\ensuremath{\mathcal{K}}}

\newcommand{\g }{\ensuremath{\mathcal{G}}}

\newcommand{\jj }{\ensuremath{\mathcal{J}}}

\newcommand{\nn }{\ensuremath{\mathcal{N}}}
\newcommand{\eee }{\ensuremath{\mathcal{E}}}

\newcommand{\LL }{\ensuremath{\mathcal{L}}}

\newcommand{\qq}{\ensuremath{\mathcal{Q}}}
\newcommand{\pp}{\ensuremath{\mathcal{P}}}

\newtheorem{theorem}{Theorem}[section]

\newtheorem{lemma}[theorem]{Lemma}

\newtheorem{proposition}[theorem]{Proposition}

\newtheorem{corollary}[theorem]{Corollary}

{\theoremstyle{definition}}

{\theoremstyle{definition}}

{\theoremstyle{definition}}

{\theoremstyle{definition}}

{\theoremstyle{definition}\newtheorem{definition}[theorem]{Definition}}

{\theoremstyle{definition}}

{\theoremstyle{definition}\newtheorem{remark}[theorem]{Remark}}

\newtheorem{conjecture}[theorem]{Conjecture}

\address{}
\email{}

\newcommand{\we}{\ensuremath{\wedge}}

\newcommand{\oti}{\ensuremath{\otimes}}

\newcommand{\he}{^{\vphantom{*}} }
\newcommand{\be}{_{\vphantom{i}} }

\newcommand{\ba}[1]{\ensuremath{\overline{#1}}}

\newcommand{\ti }[1]{\ensuremath{\widetilde{#1}}}

\newcommand{\rb }{\ensuremath{\raisebox}}

\newcommand{\tim }{\ensuremath{\times}}

\newcommand{\ee }{\ensuremath{^{\, *}}}

\newcommand{\ei }{\ensuremath{_{*}}}

\newcommand{\bop }{\ensuremath{\bigoplus\limits}}

\newcommand{\suq }{\ensuremath{\subseteq}}

\newcommand{\pe }{\ensuremath{^{\, !}}}

\entrymodifiers={+!!<0pt,\fontdimen22\textfont2>}

\def\apl#1#2#3{#1\mkern -1 mu:\mkern - 6 mu
\xymatrix@C=17pt{#2\!\ar[r]&\!#3}
}

\def\aplexp#1#2#3#4{#1\mkern -1 mu:\mkern - 6 mu
\xymatrix@C=17pt{#2\!\ar[r]^-{#4}&\!#3}
}

\def\aplcourte#1#2#3{#1\mkern -4 mu:\mkern - 8 mu
\xymatrix@C=12pt{#2\!\ar[r]&\!#3}
}

\def\aplpt#1#2#3#4{#1\mkern -4 mu:\mkern - 8 mu
\xymatrix@C=17pt{#2\!\ar[r]&\!#3#4}
}

\def\sutrgd#1#2#3{
\xymatrix@C=17pt{
0\ar[r]&#1\ar[r]&#2\ar[r]&#3\ar[r]&0
}
}

\def\sutrgdpt#1#2#3#4{
\xymatrix@C=17pt{
0\ar[r]&#1\ar[r]&#2\ar[r]&#3\ar[r]&0#4
}
}

\def\sutrgpt#1#2#3#4{
\xymatrix@C=17pt{
0\ar[r]&#1\ar[r]&#2\ar[r]&#3#4
}
}

\def\sutr#1#2#3{
\xymatrix@C=17pt{
#1\ar[r]&#2\ar[r]&#3
}
}

\def\sutrg#1#2#3{
\xymatrix@C=17pt{
0\ar[r]&#1\ar[r]&#2\ar[r]&#3
}
}

\def\sutrd#1#2#3{
\xymatrix@C=17pt{
#1\ar[r]&#2\ar[r]&#3\ar[r]&0
}
}

\def\fl{\xymatrix@C=15pt{
\ar[r]&
}}

\def\flexp#1{\smash{\xymatrix@C=15pt{
\ar[r]^-{#1}&
}}}

\def\flba{\xymatrix@C=17pt{
\ar@{|->}[r]&
}}

\def\flcourte{\xymatrix@C=10pt{
\ar[r]&
}}

\def\flgd#1#2{\xymatrix@C=17pt{#1\!
\ar[r]&\!#2
}}

\def\flgdexp#1#2#3{\xymatrix@C=17pt{#1\!
\ar[r]^-{#3}&\!#2
}}

\def\flcourtegd#1#2{\xymatrix@C=15pt{\!\!#1\!
\ar[r]&\!#2\!\!
}}

\def\flgdin#1#2{\xymatrix@C=3ex{\!\!\scriptstyle{#1}
\ar[r]&\scriptstyle{#2}\!\!\!
}}

\def\fldouble{\xymatrix@1{
\ar@{->>}[r]&
}}

\def\fle#1#2{
\xymatrix@1{
#1
\ar[r]&#2
}}

\def\flex#1#2#3{
{\xymatrix@1{
#1
\ar[r]^{#3}&#2
}}
}

\def\fledouble#1#2{
{\xymatrix@1{
#1
\ar@{->>}[r]&{#2}
}}
}

\def\flexdouble#1#2#3{
{\xymatrix@1{
#1
\ar@{->>}[r]^{#3}&{#2}
}}
}

\def\diagca#1#2#3#4#5#6#7#8{\xymatrix@1{
#1
\ar[d]_{#6}\ar[r]_{#5}&#2\ar[d]_{#7}\\
#3
\ar[r]_{#8}&#4
}}

\def\sutrois#1#2#3{
{\xymatrix@1{
#1
\ar[r]&#2
\ar[r]&#3
}}
}

\def\sutroiszerogdprime#1#2#3{
{\xymatrix@1{
0
\ar@<-0.5mm>[r]&#1
\ar@<-0.5mm>[r]&#2
\ar@<-0.5mm>[r]&#3
\ar@<-0.5mm>[r]&0
}}
}

\def\fleprime#1#2{
\xymatrix@1{
#1
\ar[r]&#2
}}

\def\sutroisnom#1#2#3#4#5{
{\xymatrix@1{
#1
\ar[r]^{#4}&#2
\ar[r]^{#5}&#3
}}
}

\def\sutroiszerogd#1#2#3{
{\xymatrix@1{
0
\ar[r]&#1
\ar[r]&#2
\ar[r]&#3
\ar[r]&0
}}
}
\def\strgdexp#1#2#3#4#5#6{
{\xymatrix@1{
0
\ar[r]&\rb{#2ex}{$#1$}
\ar[r]&\rb{#4ex}{$#3$}
\ar[r]&\rb{#6ex}{$#5$}
\ar[r]&0
}}
}

\def\sutroiszerog#1#2#3{
{\xymatrix@1{
0
\ar[r]&#1
\ar[r]&#2
\ar[r]&#3
}}
}

\def\suxtroiszerogd#1#2#3#4#5{
{\xymatrix@1{
0
\ar[r]&#1
\ar[r]^{#4}&#2
\ar[r]^{#5}&#3
\ar[r]&0
}}
}

\def\suquatre#1#2#3#4{
{\xymatrix@1{
#1
\ar[r]&#2
\ar[r]&#3
\ar[r]&#4
}}
}

\def\suxquatre#1#2#3#4#5#6#7{
{\xymatrix@1{
#1
\ar[r]^{#5}&#2
\ar[r]^{#6}&#3
\ar[r]^{#7}&#4
}}
}

\def\sucinq#1#2#3#4#5{
{\xymatrix@1{
#1
\ar[r]&#2
\ar[r]&#3
\ar[r]&#4
\ar[r]&#5
}}
}

\def\suxcinq#1#2#3#4#5#6#7#8#9{
{\xymatrix@1{
#1
\ar[r]^{#6}&#2
\ar[r]^{#7}&#3
\ar[r]^{#8}&#4
\ar[r]^{#9}&#5
}}
}

\DeclareMathOperator{\id}{id}

\DeclareMathOperator{\im}{im}

\newcommand{\Dg }{\ensuremath{\mathfrak{D}}}

\newcommand{\bw }[2]{\ensuremath{\smash[t]{\Lambda^{#1}\be#2}}}
\newcommand{\bwi }[3]{\ensuremath{\smash{\Lambda^{{#1}\vphantom{\big)}}_{{#2}^{\he}}{#3}}}}
\newcommand{\Ee }{\ensuremath{E\ee\be}}

\newcommand{\wh }[1]{\ensuremath{\widehat{#1}}}
\newcommand{\ub }[1]{\ensuremath{\underline{#1}}}
\newcommand{\ubi }{\ensuremath{\,\underline{i\be}\,}}
\newcommand{\ubj }{\ensuremath{\,\underline{j\be}\,}}
\newcommand{\ubl }{\ensuremath{\underline{\lambda\be}}}
\newcommand{\ubm }{\ensuremath{\underline{\mu\be}}}
\newcommand{\uba }{\ensuremath{\underline{\alpha \be}}}

\newcommand{\thr}{\ensuremath{\theta _{I}\he[r]}}
\newcommand{\rh}[1]{\mathcal{RH}om_{#1}\he}
\newcommand{\rhl}[1]{\mathcal{RH}om^{\,\ell\vphantom{p}}_{#1}}
\newcommand{\rhr}[1]{\mathcal{RH}om^{\,r\vphantom{p}}_{#1}}
\newcommand{\om}[1]{\ensuremath{\Omega_{A/k}^{\,#1}}}
\newcommand{\nbh}{neighbourhood}
\newcommand{\bax}{\ensuremath{\ba{\hphantom{\textrm{X}}\vphantom{\textrm{X}\he}}}
\hspace*{-2.13ex}X}
\newcommand{\eex}{\,\ensuremath{\ba{\hphantom{\textrm{X}}\vphantom{\textrm{X}\he}}}
\hspace*{-1.8ex}\eee}

\newcommand{\baxx}{\ensuremath{{}_{\ba{\hphantom{\textrm{X}}\vphantom{
{\textrm{\Large a}}}}\hspace*{-1.6ex}X}}}

\newcommand{\oox}{\ensuremath{\oo_{X}\he}}
\newcommand{\ooy}{\ensuremath{\oo_{Y}\he}}

\newcommand{\bigo}[2]{\smash{\bigotimes_{{#1}^{\he}}^{{#2}\vphantom{\big)}}}}

\author{Julien Grivaux}

\address{CNRS, LATP\\
UMR 6632\\
CMI, Universit\'{e} de Provence\\
39, rue Fr\'{e}d\'{e}ric Joliot-Curie\\
13453 Marseille Cedex 13\\
France.}

\email{jgrivaux@cmi.univ-mrs.fr}
\title{The Hochschild-Kostant-Rosenberg isomorphism for quantized analytic cycles}

\begin{document}

\begin{abstract}
In this article, we provide a detailed account of a construction sketched by Kashiwara in an unpublished manuscript concerning generalized HKR isomorphisms for smooth analytic cycles whose conormal exact sequence splits. It enables us, among other applications, to solve a problem raised recently by Arinkin and C\u{a}ld\u{a}raru about uniqueness of such HKR isomorphisms in the case of the diagonal injection. Using this construction, we also associate with any smooth analytic cycle endowed with an infinitesimal retraction a cycle class which is an obstruction for the cycle to be the vanishing locus of a transverse section of a holomorphic vector bundle.
\end{abstract}
\vspace*{1.cm}
\maketitle
\section{Introduction}
The existence of the Hochschild-Kostant-Rosenberg (HKR) isomorphism is a fundamental result both in algebraic geometry and in homological algebra. Let us recall the statement:
\begin{theorem}[\cite{HKR}]
Let $A$ be a finitely generated regular commutative algebra over a field $k$ of characteristic zero. Then for any nonnegative integer $i$, the Hochschild homology group $\emph{HH}_i \he (A)$ is isomorphic to the module $\smash{\Omega^{\,i}_{A/k}} $ of K\"{a}hler differentials of degree $i$ of $A$.
\end{theorem}
The HKR isomorphism has been generalized in the context of algebraic geometry in \cite{SW} and \cite{YE}: for any smooth quasi-projective variety over a field of characteristic zero, the derived tensor product ${\oox \lltens{}_{\oo_{X \times X}}\he \oox}$ is isomorphic in the derived category of sheaves of $\oox$-modules to the direct sum of its cohomology objects, which is $\bigoplus_{i} \, \Omega_X^i[i]$.
The same result also holds for smooth (or even singular) complex manifolds as shown in \cite{BF} and \cite{SH}, but the proof is much more involved.
\par \bigskip
We are interested here in a generalization of the analytic HKR isomorphism consisting in replacing the diagonal injection by an arbitrary closed embedding. If $(X, Y)$ is a pair of complex manifolds such that $X$ is a closed complex submanifold of $Y$\!, the derived tensor product ${\oox \lltens{}_{\ooy} \he \oox}$ is not isomorphic in general to the direct sum of its cohomology objects. This fact is the main issue of \cite{AC}, where it is proved that ${\oox \lltens{}_{\ooy} \he \oox} $ is isomorphic to $\smash[b]{\bigoplus_{i} \,\bw{i}{} N\ee_{X/Y}[i]} $ if and only if the normal bundle of $X$ in $Y$ extends to a locally-free sheaf on the first formal neighbourhood $\bax$ of $X$ in $Y$\!. More precisely, if $\mathcal{N}$ is such an extension, the authors construct a specific generalized HKR isomorphism, generally depending on $\mathcal{N}$, between $\oox \lltens{}_{\ooy} \he \oox$ and $\bigoplus_{i}\,\bw{i}{} N\ee_{X/Y}[i]$. Therefore, it appears clearly that it is necessary to quantize an analytic cycle (i.e.\ to add some additional geometric data) in order to associate with this cycle a well-defined HKR isomorphism. Quantizing the normal bundle allows to define HKR isomorphisms for the most general cycles (while dealing with smooth cycles), but the counterpart of this generality is that the space of the possible quantizations of a cycle cannot be easily handled. For instance, the following problem is raised in \cite{AC}: in the case of the diagonal injection, are the HKR isomorphisms associated with the quantizations given by the two canonical projections the same? More generally, the comparison of HKR isomorphisms associated with different quantizations of an analytic cycle is still an open problem.
\par \bigskip
In this article, our aim is to present a different construction of HKR isomorphisms associated with pairs $(X, Y)$ of complex manifolds satisfying a more restrictive condition than the aforementioned one: the normal (or conormal) exact sequence associated with the cycle $X$ has to be holomorphically split, which means in an equivalent way that the injection of $X$ into $\bax$ admits a holomorphic retraction. For any such retraction $\sigma$, the locally-free sheaf $\smash{\sigma\ee N_{X/Y}\he} $ is a quantization of $\smash{N_{X/Y}\he} $ as defined above. As far as the diagonal injection is concerned, this process is carried out in \cite{KA} (which is reproduced in \cite[chap.\ 5]{KS1}); the general case is sketched in \cite{KA}. Considering its importance, we provide a detailed account of the construction.
\par \bigskip
In this setting, analytic cycles are quantized by retractions of their first formal injection (that is the injection into their first formal neighbourhood) so that the set of possible quantizations of an analytic cycle is an affine space whose underlying vector space is $\smash{\textrm{Hom}(\, \Omega_X^1, N\ee_{X/Y})} $. With such a quantization $\sigma$ is associated a complex $\mathcal{P}_{\sigma} \he$ of coherent sheaves on $\bax$ which is quasi-isomorphic to $\oox$ and reduces to the first part of the Atiyah exact sequence when $X$ is a divisor in $Y$ (this is why we call $\mathcal{P}_{\sigma} \he$ the Atiyah-Kashiwara complex associated with $\sigma$). The sheaves defining $\mathcal{P}_{\sigma} \he$ are torsion sheaves, so that they are definitely not flat over $\ooy$. However, a remarkable fact is that $\mathcal{P}_{\sigma} \he$ can be used to compute the derived tensor product $\smash[b]{\oox \lltens{}_{\ooy} \he \oox}$ and therefore to get a specific HKR isomorphism $\Gamma_{\sigma} \he$ between $\smash[b]{\oox \lltens{}_{\ooy} \he \oox}$ and $\bigoplus_{i}\, \bw{i}{} N\ee_{X/Y}[i]$. It turns out that $\Gamma_{\sigma} \he$ is exactly the HKR isomorphism constructed in \cite{AC} associated with the quantization $\smash{\sigma\ee N_{X/Y} \he} $ of $\smash{N_{X/Y}\he} $.

\par \bigskip
Our first result provides sufficient conditions in order that two different retractions of the first formal injection of an analytic cycle define the same HKR isomorphism:
\begin{theorem}\label{11}
Let $(X,Y)$ be a pair of complex manifolds such that $X$ is a closed submanifold of $Y$ and let $\ba{j}$ be the injection of $X$ into its first neighbourhood $\bax$ in $Y$\!.
\begin{enumerate}
 \item [(1)] Assume that $N\ee_{X/Y}$ carries a global holomorphic connection. Then for any retractions $\sigma $ and $\sigma '$ of $\ba{j}$, $\pp_{\sigma }\he$ is naturally isomorphic to $\pp_{\sigma '}\he$.
\par \smallskip
\item[(2)] Let $\sigma $ and $\sigma '$ be two retractions of $\ba{j}$ such that the element $\sigma '-\sigma $ in $\emph{Hom}_{\oox }\he(\Omega ^{1}_{X}, N\ee_{X/Y})$ is an isomorphism. Then $\pp_{\sigma }\he$ is naturally isomorphic to $\pp_{\sigma '}\he$.
\end{enumerate}
 \end{theorem}
In the case of the diagonal injection, the quantizations $\pr_1 \he$ and $\textrm{pr}_{2}$ satisfy the second condition of the theorem, which gives a positive answer to the problem mentioned above.
\par \bigskip
Another important outcome of this construction is what we call the dual HKR isomorphism. To explain this notion, we consider the complex $\rh{\ooy}(\oox, \oox)$ corresponding to Hochschild cohomology in the case of the diagonal injection. This complex is well-defined up to a unique isomorphism in the bounded derived category  $D^{\textrm{b}}\be(\ooy)$ of sheaves of $\ooy$-modules, but not in $D^{\textrm{b}}\be(\oox)$. Indeed, the canonical isomorphism in $D^{\textrm{b}}\be(\ooy)$ between $\textrm{R}[\hoo_{\ooy} \he(\,*\,, \oox)]\, (\oox)$ and $\textrm{R}[\hoo_{\ooy}(\oox,\,*\,)]\, (\oox)$ is not induced in general by an isomorphism in the category $D^{\textrm{b}}\be(\oox)$. The purpose of the dual construction is to construct a specific isomorphism (the dual HKR isomorphism) between $\textrm{R}[\hoo_{\ooy}(\oox,\,*\,)]\, (\oox)$ and $\smash{\bigoplus_{i}\,\bw{i}{} N_{X/Y} \he[-i]} $ in $D^{\textrm{b}}\be(\oox)$. This is achieved by replacing $\oox$ by the dual complex $\smash{\hoo_{\oox}\he (\mathcal{P}_{\sigma} \he, \oox)}$, which is also a bounded complex of coherent sheaves on $\bax$.
\par \bigskip
The dual HKR isomorphism is a powerful tool, which has been used initially in \cite{KA} for the diagonal injection to give a functorial definition of Euler classes of coherent sheaves; it has led to a simple proof of the Grothendieck-Riemann-Roch theorem in Hodge cohomology for arbitrary proper morphisms between complex manifolds \cite{G}. We provide here another application: for any quantized analytic cycle $(X, \sigma)$ in a complex manifold $Y$\!, we construct a cohomology class $q_{\sigma} \he (X)$ in $\smash[b]{\bigoplus_{i}\,\textrm{H}^{i}\be (X, \bw{i}{} N\ee_{X/Y})} $ called the quantized cycle class of $(X, \sigma)$. We prove that this class provides an obstruction for $X$ to be defined as the vanishing locus of a transverse section of a holomorphic vector bundle on $Y$:
\begin{theorem}\label{22}
Let $(X,\sigma )$ be a quantized analytic cycle of codimension $r$ in $Y$ and assume that there exists a couple $(E,s)$ such that
\begin{enumerate}
 \item [(1)] $E$ is a holomorphic vector bundle of rank $r$ on $Y$ .
\item[(2)] $s$ is a holomorphic section of $E$ vanishing exactly on $X$, and $s$ is transverse to the zero section.
\item[(3)] The locally-free $\oox$-modules $E\oti_{\ooy}\he\oo\baxx$ and $\sigma \ee\be N_{X/Y}\he$ are isomorphic.
\end{enumerate}
Then $q_{\sigma} \he (X)=1$.
\end{theorem}
For the diagonal injection, it follows from the results of \cite{MA}, \cite{RA1}, \cite{G} and \cite{RA2} that $q_{\,\textrm{pr}_1}(\Delta_X)$ is the Todd class of $X$. Up to the author's knowledge, the quantized cycle class $q_{\sigma} \he (X)$ has not yet appeared in the literature, and it would be interesting to compute it in purely geometrical terms.
\par \bigskip
To conclude this introduction, let us discuss the link of this construction with the generalized Duflo isomorphism. The aim of this isomorphism is to understand precisely how the HKR isomorphism between the algebras $\textrm{Ext}\ee_{\oo_{X \times X}}(\oox, \oox)$ and $\textrm{H}\ee(X, \bw{*} \, TX)$ fails to be multiplicative. After the seminal work \cite{KO}, the following result (conjectured in \cite{C1}) was proved:
\begin{theorem}[\cite{CV}]
For any complex manifold $X$, if $\Gamma$ denotes the standard \emph{HKR} isomorphism between $\emph{Ext}\ee_{\oo_{X \times X}}(\oox, \oox)$ and  $\textrm{H}\ee(X, \bw{*} \, TX)$, then $(\emph{td}(X))^{-1/2} \lrcorner \, \Gamma$ is a ring isomorphism.
\end{theorem}
For general cycles, the algebra $\smash[t]{\textrm{Ext}\ee_{\ooy}(\oox, \oox)} $ is no longer graded commutative, and its structure is the object of current active research (see the program initiated in \cite{CCT1}). The quantized cycle class $\smash{q_{\sigma} \he (X)} $, which generalizes the Todd class for arbitrary quantized analytic cycles, is likely to play a part in the understanding of this algebra.
\par \bigskip
Let us describe more precisely the outline of this article. After a preliminary section (\S~\ref{SectAlgExt}), it is divided into  three main parts: the local construction of HKR isomorphisms is carried out in \S~\ref{LocCx}, these results are globalized in \S~\ref{Atiyah Global} and provide an application in \S~\ref{cycle} where we construct and study the quantized cycle class. We now turn to the specific organization of each part.
\par \medskip
In \S~\ref{cup}, we recall some elementary constructions in exterior algebra such as contraction morphisms and Koszul complexes, mainly to fix sign conventions. In \S~\ref{DgAlg}, an abstract construction on dg-algebras is performed, the aim of which is to provide a general setting for Atiyah-Kashiwara complexes.
\par \medskip
At the beginning of \S~\ref{LocCx}, we define specific notation for the derived functors of the functor Hom and for the tensor product, since they cannot be derived as bifunctors in our setting. In \S~\ref{LocHKR} are defined the Atiyah-Kashiwara complex (Definition \ref{DefUnHKR}) together with the dual Atiyah-Kashiwara complex (Definition \ref{DefDeuxHKR}); and in Propositions \ref{PropUnHKR} and \ref{PropDeuxHKR} we establish the corresponding local HKR isomorphisms. The proofs we give here are bound to extend naturally to a global setting. In Proposition \ref{PropDeuxBisHKR}, we compare in a weak sense the HKR and dual HKR isomorphisms. The argument of this proof will be used anew in the proof of Theorem \ref{BaProCyClThUn} (which is Theorem \ref{22} in this introduction). In \S~\ref{SectionArikinCalda}, the construction performed in \S~\ref{LocHKR} is compared to the construction of \cite{AC} in the local case, and both are shown to be compatible in Proposition \ref{PropUnArinkinCalda}. In \S~\ref{SousSecAKComplexes}, Proposition \ref{PropTroisHKR} provides conditions to construct naturally automorphisms of Atiyah-Kashiwara complexes, and is the local version of Theorem \ref{11}.
\par \medskip
The next part (\S~\ref{Atiyah Global}) deals with the complex analytic case. In the first section (\S~\ref{anhkr}), the results of \mbox{\S~\ref{LocHKR}} and \mbox{\S~\ref{SousSecAKComplexes} }are stated in a global setting, Propositions \ref{PropUnAnalyticHKR}, \ref{HKR2} and Theorem \ref{PropDeuxAnalyticHKR} (which is Theorem \ref{11} in this introduction) extending Propositions \ref{PropUnHKR}, \ref{PropDeuxHKR} and \ref{PropTroisHKR} respectively. In \S~\ref{TwistedCase}, we explain how to twist Atiyah-Kashiwara complexes by extension classes. In Proposition \ref{PropUnTwisted}, we prove that two Atiyah-Kashiwara complexes associated with different retractions become isomorphic after twisting by extension classes depending on the Atiyah class of the conormal bundle $\smash{N\ee_{X/Y}} $. In Theorem \ref{PropDeuxTwisted} we recall (in slightly more general terms) the principal result of \cite{AC} and we prove in Theorem \ref{PropTroisTwisted} that, when the cycle admits an infinitesimal retraction, the HKR isomorphisms of \cite{AC} associated with arbitrary quantizations of the normal bundle are again twisted HKR isomorphisms in our sense. In the case of the canonical quantization associated with a retraction, we obtain the compatibility of HKR isomorphisms (this globalizes proposition \ref{PropUnArinkinCalda}). The aim of \S~\ref{comparison} is to study and carefully compare twisted HKR isomorphisms (and so to compare HKR isomorphisms associated with different retractions, thanks to Proposition \ref{PropUnTwisted}). We give some results in particular cases, namely when the twist are obtained by tensorization with holomorphic line bundles on $\bax$ (Theorem \ref{ThUnCompHKR}), and then for general extension classes when only the last but one term of each Atiyah-Kashiwara complex is twisted (Theorem \ref{ModuleLunaireUn}). As a corollary, we deduce in Theorem \ref{ModuleLunaireDeux} the general comparison theorem between HKR isomorphisms associated with different retractions for cycles of codimension two. We are led to propose a conjecture for the general case (Conjecture \ref{conj}).
\par \medskip
The last part (\S~\ref{cycle}) deals with the quantized cycle class. In \S~\ref{cycle1}, using the dual HKR isomorphism, we define this cycle class and compute it in specific cases.  In Theorem \ref{BaProCyClThUn} (which is Theorem \ref{22} in this introduction), we prove that the quantized cycle class is one when the cycle $X$ is the zero locus of a transverse section of a holomorphic vector bundle on $Y$ satisfying a compatibility condition with the retraction $\sigma$. In Theorem \ref{BaProCyClThDeux}, we obtain that the class $q_{\,\textrm{pr}_1}(\Delta_X)$ is the Todd class of $X$; this is equivalent to the main result of \cite{G}. Finally, we deal with the divisor case in theorem \ref{divisor}. Preliminary constructions for  \S~\ref{kash} are carried out in \S~\ref{six}: if $j$ denotes the injection of the cycle $X$ into $Y$\!, we study the right and left adjoints $j\ee$ and $j\pe$ of the direct image functor $j\ei$ operating on the corresponding derived categories. In \S~\ref{kash}, following \cite[chap.\ 5]{KS1} for the diagonal injection, we establish in Theorem \ref{KashIsoThDeux} that the natural isomorphism between $j\ee\be j\ei\he\oox \, \lltens{}_{\oox}\he  \omega _{X/Y}\he$ and $j\pe\be j\ei\he\oox$ obtained using the local cycle class of $X$ in $Y$ is given via the HKR isomorphisms by contraction with the quantized cycle class $q_{\sigma} \he (X)$.
\par\bigskip
\textbf{Acknowledgments.} I wish to thank Pierre Schapira who has encouraged me all along, and also Damien Calaque and Richard Thomas for useful conversations and comments.
\section{Preliminary constructions}\label{SectAlgExt}
\subsection{Duality and cup-product}\label{cup}
Let $r$ be a positive integer, $A$ be a commutative $k$-algebra over a field $k$ of characteristic zero, and $E$ be a free $A$-module of rank $r$. In this section, all tensor and exterior products are taken over $A$.
\begin{definition}\label{DefUnSecAlgExt}
 For any nonnegative integers $p$ and $q$, we denote by $\apl{W_{p,\,q}}{\bwi{p+q}{}E}{\bw{p}{E}\oti\!\bw{q}{E}}$
the transpose of the cup-product map from $\bw{p}{\Ee}\oti\bw{q}{\Ee}$ to $\bw {p+q}{\Ee}$ multiplied by $\dfrac{p! \, q!}{(p+q)!}$.
\end{definition}
\par \medskip
It is possible to give another natural definition of $W_{p,q}$ as follows: for any nonnegative integer $n$, let $\mathfrak{S}_n \he$ denote the symmetric group with $n$ letters, and let $\apl{\varepsilon}{\mathfrak{S}_n \he}{\{-1, \,1\}}$ be the signature morphism. We define the symmetrization and antisymmetrization maps
$\apl{\mathfrak{a}_n \he}{\bigo{}{n}E}{\bw{n \be}{E}}$ and $\apl{\mathfrak{s}_n \he}{\bw{n \be}{E}}{\bigo{}{n}E}$ by the formulae below:
\begin{equation} \label{EqUnPrelim}
\begin{cases}
\mathfrak{a}_n \he \, (v_{1}\he\oti\dots\oti v_{n}\he)=v_{1}\he\we\dots\we v_{n}\he \\
\mathfrak{s}_n \he \, (v_{1}\he\we\dots\we v_{n}\he)=\dfrac{1}{n!} \, \ds\sum_{\sigma \in \mathfrak{S}_n \he} \varepsilon(\sigma)\, v_{\sigma(1)}\he\oti\dots\oti v_{\sigma(n)}\he\\
\end{cases}
\end{equation}
A straightforward computation shows that
\begin{equation} \label{EqDeuxBisSecAlgExt}
W_{p,q}\he(v_1 \he \we \ldots \we v_{p+q} \he)= \dfrac{p! \, q!}{(p+q)!} \ds \sum_{\sigma} \varepsilon(\sigma) \, (v_{\sigma(1)}\he\we\dots\we v_{\sigma(p)}\he) \, \oti \, (v_{\sigma(p+1)}\he\we\dots\we v_{\sigma(p+q)}\he)
\end{equation}
where $\sigma$ runs through all $(p\,$-$q)$ shuffles, which implies that $W_{p,q}=(\mathfrak{a}_{p} \oti \mathfrak{a}_{q}) \circ \mathfrak{s}_{p+q}.$
\begin{definition}\label{DefUnBisSecAlgExt}
 For any nonnegative integers $m$, $p$, $k$ and any $\phi $ in $\textrm{Hom}(\bw{p}{E}, \bw{k}{E})$, we define $t_{k,\,p} ^{\,m}(\phi )$ in $\textrm{Hom}(\bw{p+m}{E}, \bw{k+m}{E})$ by the composition
\[
\xymatrix@C=20pt{t_{k,p} ^{\,m}(\phi )\,:\,\bw{p+m}{E}\ar[rr]^-{W_{p,\,m}\he}&&\bw{p}{E}\oti \bw{m}{E}\ar[rr]^-{\phi\, \oti\,\id}&&\bw{k}{E}\oti \bw{m}{E}\ar[r]^-{\we}&\bw{k+m}{E.}
}
\]
\end{definition}
The translation operator $\apl{t_{k,\,p}^{\,m}(\phi )}{\textrm{Hom}(\bw{p}{E}, \bw{k}{E})}{\textrm{Hom}(\bw{p+m}{E}, \bw{k+m}{E})}$ satisfies the following important property:
\begin{lemma}\label{LemUnBisSecAlgExt}
For any nonnegative integers $m$, $p$, $k$ such that $k\geq p$ and for any $a$ in $\bw{k-p}E$, we have $t_{k, \,p}^{\,m}(a\we\,.\,)=a\we\,.$
\end{lemma}
\begin{proof}
 By (\ref{EqDeuxBisSecAlgExt}), for any $e_{1}\he,\dots, e_{p+m}\he$ in $E$, we have
\begin{align*}
t_{k, \, p}^{\,m}(a\we\,.\,)\,(e_{1}\he\we\dots\we e_{p+m}\he)&=\dfrac{p! \, m!}{(p+m)!} \ds \sum_{\sigma} \varepsilon(\sigma) \, \, a \we e_{\sigma(1)}\he\we\dots\wedge e_{\sigma(p)}\he \we e_{\sigma(p+1)}\he\we\dots\wedge e_{\sigma(p+m)}\he\\
&=a\we e_{1}\he\we\dots\we e_{p+m}\he.
\end{align*}

\end{proof}

For any positive integer $p$, any vector $v$ in $\bw{}{E}$ defines two endomorphisms $\ell_v$ and $r_v$ of $\bw{}{E}$ given by $\ell_v(x)=v \we x$ and $r_v(x)=x \we v$.  The map ${}^t r_v\,$ (resp.\ ${}^t \ell_v\,$) is by definition the \textit{left} (resp.\ \textit{right}) contraction by $v$; it is an endomorphism of $\bw{}{\Ee}$ denoted by $\smash{\xymatrix@C=17pt{\phi\ar@{|->}[r]&v\lrcorner\, \phi}} $ (resp.\,$\smash{\xymatrix@C=17pt{\phi\ar@{|->}[r]&\phi \, \llcorner\, v).}} $ The left (resp.\ right) contraction morphism endows $\bw{}{\Ee}$ with the structure of a left (resp.\ right) $\bw{}{E}$-module.
\par \bigskip
There are two duality isomorphisms $D^{\ell}\be$ and $D^{r}\be$ from $\bw{}{E} \oti \textrm{det}\, E\ee$ to $\bw{}{\Ee}$ given by \begin{align}
D^{\ell}(v \oti \xi)=v \lrcorner \, \xi \qquad \textrm{and} \qquad D^{r}(v \oti \xi)=\xi \, \llcorner \, v
\end{align}
Remark that $D^{\ell}\be$ (resp.\ $D^{r}\be$) is an isomorphism of left (resp.\ right) $\bw{}{E}$-modules.
\par\bigskip
We pause for a moment in order to discuss sign conventions concerning contraction morphisms. Let $\smash{\Delta_r=\{(i,j) \in \N^2 \, \, \textrm{such that}\, \,  i+j \leq r\}} $. For any sign function $\smash{\apl{\chi}{\Delta_r \he}{\Z_2 \he,}} $ we can consider the left (resp.\ right) twisted contraction morphism from $\bw{}{E} \oti \bw{}{\Ee}$ (resp.\ $\bw{}{\Ee} \oti \bw{}{E}$) to $\bw{}{\Ee}$ defined on homogeneous elements by the formula
\[
\begin{cases}
v \, \lrcorner_{\chi} \he \, \phi=\chi[\textrm{deg(}v\textrm{)}, r-\textrm{deg(}\phi\textrm{)}] \, \,  v\, \lrcorner \, \phi \\
\phi \, \llcorner_{\chi} \he \, v=\chi[\textrm{deg(}v\textrm{)}, r-\textrm{deg(}\phi\textrm{)}] \, \,  \phi\, \llcorner \, v
\end{cases}
\]
A routine computation shows that the left (resp.\ right) twisted contraction by a sign function $\chi$ defines a left (resp.\ right) action of $\bw{}{E}$ on $\bw{}{\Ee}$ if and only if $\chi$ is one of the four following functions:
\par\smallskip
\begin{enumerate}
\item[--] $\,\chi(p,q)=1$
\par \medskip
\item[--] $\,\chi(p,q)=(-1)^p \be$
\item[--] $\,\chi(p,q)=(-1)^{\frac{p(p+1)}{2}+pq} \be$
\item[--] $\,\chi(p,q)=(-1)^{\frac{p(p+1)}{2}+pq+p} \be$
\end{enumerate}
\par \medskip
Therefore there are \textit{four} different sign conventions for a left (resp.\ right) action of $\bw{}{E}$ on $\bw{}{\Ee\!.}$
\par\bigskip
We end this section with Koszul complexes. Let $M$ be an \mbox{$A$-module} and let $\phi $ be an $A$-linear form on $M$. The \emph{Koszul complex} $L(M,\phi )$ is the exterior algebra $\bwi{}{A}{M}$ endowed with the differential $\delta $ of degree $-1$ given for any positive integer $p$ by
\[
\delta _{p}\he(m_{1}\he\we\dots\we m_{p}\he)=\sum_{i=1}^{p}\,(-1)^{i-1}\be\phi (m_{i}\he)\,\,
m_{1}\he\we\dots\we m_{i-1}\he\we m_{i+1}\he\we\dots\we m_{p}\he.
\]
If $x_{1}\he,\dots,x_{k}\he$ are elements in $A$, we recover the classical Koszul complex associated with the $x_{i}\he$'s by taking $M=A^{k}\be$ and $\phi (a_{1}\he,\dots a_{k}\he)=\sum_{i=1}^{k}x_{i}\he a_{i}\he$.
\par\medskip
Assume now that $M$ is free of finite rank $r$, and consider $M$ as the dual of $M\ee\be $. In this case, $\delta $ is exactly the right contraction by $\phi $ acting on $\bwi{}{A}{M}$. Using the standard sign convention for $\textrm{Hom}$ complexes (see for instance \cite{KS1} Remark 1.8.11 and \cite{D} Remark 1.1.11), the differential $\delta \ee\be$ of $L\ee\be$ is given for any nonnegative integer $p$ by
\[
\delta \ee_{p}=(-1)^{p+1}\be\,\phi \,\we.=-\,(\,.\we\phi )
\]
Thus, the right duality morphism $\aplexp{D^{r}\be}{\bw{}{M\ee\be\oti\det M}}{\bw{}{M}}{\sim}$ induces an isomorphism
\begin{equation}\label{EqTroisPrelim}
 (L\ee\be,\delta \ee\be)\simeq (L,-\delta ) \oti_{A}\he \det M\ee\be[-r].
\end{equation}

\subsection{Extensions and dg-algebras}\label{DgAlg}
Let $A$ be a (non necessarily commutative) unitary algebra over a field $k$ of characteristic zero, $I$ be a $A\oti A^{\textrm{op}}\be$-module, and $B$ be the trivial $k$-extension of $A$ by $I$. This means that $B=I\oplus A$, endowed with the algebra structure defined by
\[
(i,a)\,.\,(i',a')=(i a'+a i',aa').
\]
\par\smallskip
Let us take a dg-algebra $(\aaa, d\,)$ over $k$ concentrated in positive degrees, whose differential has degree $-1$, and satisfying the following compatibility condition:
\begin{equation}\label{EqUnDgAlg}
 \emph{The truncated dg-algebra}\ \smash[t]{\flgdexp{\aaa_{1}\he}{\aaa_{0}\he}{d_{1}\he}}\  \emph{is  isomorphic to}\
\smash[t]{\flgdexp{B}{A.}{\pr_{2}\he}}
\end{equation}
We denote by $|a|$ the degree of an homogeneous element $a$ in $\aaa$.
\begin{definition}\label{DefUnDgAlg}
 Let $\bb$ denote the graded module $\bop\nolimits_{k\ge 1}\aaa_{k}\he$, where each $\aaa_{k}\he$ sits in degree $k-1$.
\begin{enumerate}
 \item [(1)] For any homogeneous elements $a$ and $a'$ in $\bb$, we put
\[
a*a'=a\,.\,da'+(-1)^{|a|+1}\be da\,.\, a'+(-1)^{|a|}\be\, da\,.\,1_{B}\he\,.\,da'
\]
where $1_{B}\he=(0, 1_{A})$ is the the unit of $B$ considered as an element of $\aaa_{1}\he$.
\item[(2)] We endow $\bb$ with a differential $\wh{d}$ of degree $-1$ given for any positive integer $k$ by $\wh{d}_{k}\he=k\,d_{k+1}$.
\end{enumerate}
\end{definition}
Remark that via the isomorphism between $\aaa_{1}\he$ and $B$, the product $*$ from $\aaa_{1}\he\oti_{k}\he\aaa_{1}\he$ to  $\aaa_{1}\he$ is exactly the product in the algebra $B$.
\begin{proposition}\label{PropUnDgAlg}
 For any dg-algebra satisfying \emph{(\ref{EqUnDgAlg})}, $(\bb,*,\wh{d} \,)$ is a dg-algebra.
\end{proposition}
\begin{proof}
 This is proved by direct computation. Let us prove for instance that $\wh{d}$ satisfies Leibniz rule, and leave to the reader the associativity of $*$. We take two homogeneous elements $b$ and $b'$ in $\bb$ of respective degrees $k$ and $k'$. Then
\begin{align*}
 \wh{d}_{k+k'}\he(b*b')&=(k+k')\,d_{k+k'+1}\he(b*b')=(k+k')\,db\,.\,db'\\
&=k \,(d_{k+1}\he b*b')+(-1)^{k}\be\, k'\,(b*d_{k'+1}\he b')=\wh{d}_{k}\he b*b'+(-1)^{|b|}\be\, b*\wh{d}_{k'}\he b'.
\end{align*}
\end{proof}
It follows from this result that all the $\aaa_{k}\he$'s are naturally $B \oti_{k}\he B^{\,\textrm{op}}$-modules. Besides, $A$ can be endowed with the structure of a $ B \oti_{k}\he B^{\,\textrm{op}}\be$-module, and there is a natural $B \oti_{k}\he B^{\,\textrm{op}}\be$-linear morphism $\apl{\pi }{\bb}{A}$ obtained via the composition $\smash{\xymatrix@C=17pt{\bb\ar[r]&\bb_{0}\he\simeq\aaa_{1}\ar[r]^-{\smash[b]{d_{1}\he} }&\aaa_{0}\simeq A.}} $
Besides, the diagram below
\[
\xymatrix@C=30pt@R=20pt{
\bb\oti_{B}\he\bb\ar[r]^-{*}\ar[d]_-{\pi \oti\pi }&\bb\ar[d]^-{\pi }\\
A\oti_{B}\he A\ar[r]&A
}
\]
is commutative, the bottom line being given by $\xymatrix@C=17pt{a_{1}\he\oti a_{2}\he\ar@{|->}[r]&a_{1}\he a_{2}\he.}$
\par\medskip
The situation is more comfortable in the commutative case, i.e.\ when $A$ is commutative, $I$ is a $A$-module (hence a $\smash{A\oti_{k}\he A^{\textrm{op}}\be} 	$-module), and $\aaa$ is graded-commutative. In that case, $\bb$ is also graded-commutative and all the $\aaa_{k}\he$'s are endowed with a $B$-module structure commuting with the differential $\wh{d}$.
\par\medskip
In the main example of application, which we describe now, we assume $A$ to be commutative, and we take for $\aaa$ the exterior algebra $\bwi{}{A}{B}$ endowed with the Koszul differential given by the $A$-linear form $\apl{\pr_{2}\he}{B}{A.}$ Thus, for any positive integer $k$,
\[
d_{k}(b_{1}\he\we\dots\we b_{k}\he)=\sum_{i=1}^{k}(-1)^{i-1}\pr_{2}\he(b_{i})\,\,
b_{1}\he\we\dots\we b_{i-1}\he\we b_{i+1}\he\we\dots\we b_{k}\he.
\]
Besides, via the isomorphism
\begin{align}
 \bwi{k}{\!A}{I}\oplus\bwi{k-1}{\!A}{I}\ &\xymatrix@C=17pt{\ar[r]^-{\sim}&}\ \bwi{k}{\!A}{B}\label{EqDeuxDgAlg}\\
(\ubi\,,\,\ubj)\hphantom{\bwi{k}{aaa}{\bb}} &\xymatrix@C=17pt{\ar@{|->}[r]&}\ \ubi+1_{B}\he\we \ubj \notag
\end{align}
the differential $\apl{d_{k}\he}{\bwi{k}{\!A}{B}}{\bwi{k-1}{\!A}{B}}$ is obtained as the composition
\[
\sutr{\bwi{k}{\!A}{B}}{\bwi{k-1}{\!A}{I}}{\bwi{k-1}{\!A}{B.}}
\]
Thus $\aaa$ (considered as a complex of $A$-modules) is exact. Via the isomorphism (\ref{EqDeuxDgAlg}),
the product $*$ has the following explicit form:
\begin{align}
 *\,:\,\,\bigl( \bwi{k+1}{\!A}{I}\oplus\bwi{k}{\!A}{I}\bigr)\oti_{k}\he\bigl( \bwi{l+1}{\!A}{I}\oplus\bwi{l}{\!A}{I}\bigr)&\xymatrix@C=17pt{\ar[r]&}\bwi{k+l+1}{\!A}{I}\oplus\bwi{k+l}{\!A}{I}\label{EqTroisDgAlg}\\
(\ubi_{\,1}\he\,,\,\,\ubj _{\,1}\he)\oti(\ubi _{\,2}\he\,,\,\,\ubj _{\,2}\he)\hphantom{iiiiiiiiii}&\xymatrix@C=17pt{\ar@{|->}[r]&}(\ubi _{\,1}\he\we\ubj _{\,2}\he+(-1)^{k}\be\ubj _{\,1}\he\we\ubi _{\,2}\he\,,\,\,\ubj _{\,1}\he\we\ubj _{\,2}\he).\notag
\end{align}
In particular, the $B$-module structure on $\bwi{k}{\!A}{B}$ is given by the formula
\begin{equation}\label{EqQuatreDgAlg}
 (i,a)*(\ubi _{\,1}\he,\ubj _{\,1}\he)=(a\,\ubi _{\,1}+i\we\ubj _{1}\he,\,a\ubj _{\,1}\he).
\end{equation}\
\section{Atiyah complexes (I)}\label{LocCx}
In this section we assume that $A$ is a commutative algebra with unit over a field $k$ of characteristic zero, and we adopt the notations of \S~\ref{DgAlg}, except that from now on we use the \textit{cohomological} grading for complexes, which means that all differentials are of degree $+1$.
\par \bigskip
We also introduce extra notation for derived functors. Let $R$ be a commutative $k$-algebra, $M$ be an $A$-module, and assume that $A$ is a quotient of $R$. We consider the following four functors, from $\textrm{Mod}(R)$ to $\textrm{Mod}(A)$ for the three first ones and from $\textrm{Mod}(R)^{\textrm{\,op}}\be$ to $\textrm{Mod}(A)$ for the last one:
\[
S\fl M\oti_{R}\he S,\quad S\fl S\oti_{R}M,\quad S\fl \textrm{Hom}_{R}\he(M,S),\quad S\fl \textrm{Hom}_{R}\he(S,M)
\]
The associated derived functors are denoted by
\[
S\fl M\,\lltensr{}_{\!R}\he\,S,\quad S\fl S\,\lltensl{}_{\!R}\he\,M,\quad S\fl \textrm{RHom}^{r}_{R}(M,S),\quad S\fl \textrm{RHom}^{\ell}_{R}(S,M)
\]
Of course, these functors can be defined for any $M$ in $D^{-}\be(A)$ for the three first ones and for any $M$ in $D^{+}(A)$ for the last one.
\par \medskip
There is a slightly subtle point behind these definitions: for any elements $M$, $N$ in $D^{-}\be(A)$, there is a canonical isomorphism between $M\,\lltensr{}_{\!R}\he\,N$ and $M\,\lltensl{}\he_{\!R}\,N$ in $D^{-}\be(R)$, but this isomorphism is not a priori induced by an isomorphism in $D^{-}\be(A)$.
The same thing happens with the isomorphism between $\textrm{RHom}^{r}_{R}(M,N)$ and $\textrm{RHom}^{\ell}_{R}(M,N)$ in $D^{+}\be(R)$.
\subsection{HKR isomorphisms for regular ideals}\label{LocHKR}
Let $I$ be a free $A$-module of finite rank $r$. The construction performed in \S~\ref{DgAlg} allows to make the following definition:
\begin{definition}\label{DefUnHKR}
 The \emph{Atiyah-Kashiwara} (AK) complex associated with $I$ is the complex of $B$-modules
\[
P\,:\,\,\xymatrix@C=35pt{0\ar[r]&\bwi{r+1}{\!A}{B}\ar[r]^-{rd_{r+1}\he}&\bwi{r}{\!A}{B}\ar[r]^-{(r-1)d_{r+1}\he}&\ \cdots\ \ar[r]^-{d_{2}\he}&B\ar[r]&0}
\]
where $B$ is in degree $0$.
\end{definition}
There is a quasi-isomorphism $\xymatrix@C=17pt{P\ar[r]^-{\sim}&A}$ in $\textrm{Mod}(B)$. As a complex of $A$-modules, $P$ splits as the direct sum of $A$ and of a null-homotopic complex.
\par\smallskip
Let us now take a commutative $k$-algebra $C$ with unit as well as a regular ideal $J$ in $C$ of length $r$.
If $(j_{1}\he,\dots,j_{r}\he)$ is a regular sequence defining $J$, then $J/J^{2}$ is a free $C/J$-module of rank $r$, a basis being given by the classes of the elements $j_1, \ldots, j_r$. Then, if we put $A=C/J$ and $I=J/J^{2}$, we see that $C/J^{2}$ is a $k$-extension of $A$ by $I$ via the Atiyah exact sequence
\begin{equation}\label{EqUnHKR}
\sutrgdpt{J/J^{2}\be}{C/J^{2}\be}{C/J}{.}
\end{equation}
If this exact sequence splits over $A$, the algebra $C/J^{2}$ is isomorphic (in a non-canonical way) to the trivial $k$-extension of $A$ by $I$, so that we can identify $C/J^{2}\be$ with $B=I\oplus A$ after the choice of a splitting of (\ref{EqUnHKR}).
\par \medskip
\begin{proposition}[HKR isomorphism, local case]\label{PropUnHKR}
Let $C$ be a commutative $k$-algebra with unit and $J$ be a regular ideal of $C$ such that the associated Atiyah sequence \emph{(\ref{EqUnHKR})} splits. If we choose an isomorphism between $C/J^{2}$ and $B$, the quasi-isomorphism $\smash{\xymatrix@C=17pt{P\ar[r]^-{\sim}&A}} $ in \emph{Mod}\,$(C)$ induces isomorphisms
\[
\xymatrix@C=25pt{A\,\lltensr{}_{C}\he A&A\,\lltensr{}_{C}\he P\ar[l]_-{\sim}\ar[r]^-{\sim}&A\oti_{C}\he P\simeq\bop_{i=0}^{r}\bwi{i}{\!A}{I}[i]}
\]
\vspace*{-3ex}
\[
\xymatrix{\ar[r]^-{\sim}
\emph{RHom}^{\ell}_{C}(A,A)&\emph{RHom}^{\ell}_{C} (P,A)&\emph{Hom}_{C}\he(P,A)\simeq\bop_{i=0}^{r}\bwi{i}{\!A}{I\ee\be}[-i]\ar[l]_-{\sim}
}
\]
in the bounded derived category $D^{\emph{b}}(A)$, where $I\ee\be=\emph{Hom}_{\!A}\he(I,A)$.
\end{proposition}
\begin{proof}
For any element $c$ in $C$, we denote by $\ba{c\he}$ the class of $c$ in $B$. We also denote by $(e_{1}\he,\dots,e_{r}\he)$ the canonical basis of $k^{r}\be$. If $(j_{1}\he,\dots,j_{r}\he)$ is a regular sequence defining the ideal $J$, the Koszul complex $L=(C\oti_{k}\he\bw{}{k^{r}\be},\delta )$ associated with $(j_{1}\he,\dots,j_{r}\he)$ is a free resolution of $A$ over $C$.
For any  nonnegative integer $p$, we define a map $\smash{\apl{\gamma _{-p}\he}{L_{-p}\he}{P_{-p}\he}} $ by the formula
\[
\gamma _{-p}\he(c\oti e_{l_{1}}\he\we\dots\we e_{l_{p}}\he)=\ba{c\he}*(1_{B}\he\we\ba{j\he}_{l_{1}}\he\we\dots\we \ba{j\he}_{l_{p}}\he).
\]
where the product $*$ is defined in \S\ \ref{DgAlg}.
The map $\gamma _{-p}\he$ is obviously $C$-linear, and if $p$ is positive,
\begin{align*}
 \wh{d}_{-p}\he\circ\gamma _{-p}\he(c\oti e_{l_{1}}\he\we\dots\we e_{l_{p}}\he)&=
 \wh{d}_{-p}\he(\,\ba{c\he}*1_{B}\he\we\ba{j\he}_{l_{1}}\he\we\dots\we \ba{j\he}_{l_{p}}\he)\\
&=\ba{c\he}*\wh{d}_{-p}\he(1_{B}\he\we\ba{j\he}_{l_{1}}\he\we\dots\we \ba{j\he}_{l_{-p}}\he)\\
&=p\,\ba{c\he}* (\, \ba{j\he}_{l_{1}}\he\we\dots\we \ba{j\he}_{l_{p}})\he
\intertext{and}\\[-6ex]
\gamma _{-(p-1)}\he\circ\delta _{-p}\he(c\oti e_{l_{1}}\he\we\dots\we e_{l_{p}}\he)&=\gamma _{-(p-1)}\he\Bigl(\, \sum_{i=1}^{p}(-1)^{i-1}\be c\,j_{l_{i}}\he\oti e_{l_{1}}\he\we\dots\we e_{l_{i-1}}\he\we e_{l_{i+1}}\he\we\dots\we e_{l_{p}}\he\Bigr)\\[-2ex]
&=\sum_{i=1}^{p}(-1)^{i-1}\be \,\ba{c\he}*\bigl[ \,\ba{j}_{l_{i}}\he
*(1_{B}\he\we\ba{j}_{l_{1}}\he\we\dots\we\ba{j}_{l_{i-1}}\he\we\ba{j}_{l_{i+1}}\he\we\dots\we\ba{j}_{l_{p}}\he)\bigr]\\[-1ex]
&=\sum_{i=1}^{p}(-1)^{i-1}\be \,\ba{c\he}*(\, \ba{j}_{l_{i}}\we\ba{j}_{l_{1}}\we\dots\we\ba{j}_{l_{i-1}}\we\ba{j}_{l_{i+1}}\we\dots\we\ba{j}_{l_{p}})\\
&=p\,\ba{c\he}*(\, \ba{j}_{l_{1}}\we\dots\we\ba{j}_{l_{p}}).
\end{align*}
Thus $\apl{\gamma }{L}{P}$ is a morphism of complexes. Hence we get two commutative diagrams
\[
\xymatrix@C=25pt@R=1pt{
&A\,\lltensr{}_{C}\he P\ar[dl]_-{\sim}\ar[r]&A\oti_{C}\he P\\
A\,\lltensr{}_{C}\he A&&\\
&A\,\lltensr{}_{C}\he L\ar[ul]^-{\sim}\ar[uu]_-{\id\oti\gamma }\ar[r]^-{\sim}&A\oti_{C}\he L\ar[uu]_-{\id\oti\gamma }
}\quad
\xymatrix@C=25pt@R=8pt{
&\textrm{RHom}^{\ell}_{C} (P,A)\ar[dd]^-{.\,\circ\,\gamma }_-{\sim}&\textrm{Hom}_{C}\he(P,A)\ar[l]\ar[dd]^-{.\,\circ\,\gamma }\\
\textrm{and\quad RHom}\smash{^{\ell}_{C}} (A,A)\ar[ur]^-{\sim}\ar[dr]_-{\sim}&&\\
&\textrm{RHom}^{\ell}_{C} (L,A)&\textrm{Hom}_{C}\he(L,A)\ar[l]_-{\sim}
}\hspace*{-5pt}
\]
\par\medskip
The vertical right arrow in the first diagram is the map from $A\oti_{k}\he\bw{}{k^{r}\be}$ to $\bwi{}{\!A}{I}$ obtained by mapping each vector $e_k$ to $\ba{j}_k$, hence is an isomorphism. The dual of this map over $A$ is precisely (up to sign) the vertical right arrow in the second diagram, so that it is an isomorphism too. This finishes the proof.
\end{proof}
We now recall Kashiwara's construction of the dual HKR isomorphism. For any free $A$-module $I$ of finite rank, we denote by $\theta _{I}\he$ its top exterior power.
\begin{definition}\label{DefDeuxHKR}
 If $I$ is a free module of rank $r$ and $P$ is the associated AK complex, the \emph{dual \emph{AK} complex} $Q$ is the complex of $B$-modules defined by $Q=\textrm{Hom}_{A}\he(P,\theta _{I}\he[r])$ with a specific sign convention: the differential of $Q$ is $(-1)^r$ times the differential of $\textrm{Hom}_{A}\he(P,\theta _{I}\he[r])$.
\end{definition}
To describe $Q$, notice that for every integer $p$ between $0$ and $r-1$, there is an isomorphism of $B$-modules
\begin{equation} \label{dualityfirst}
\bwi{r-p}{\!A}{B}\simeq \textrm{Hom}_{A}\he(\bwi{p+1}{\!A}{B},\theta _{I}\he)
\qquad \quad
(\ub{u}, \ub{v}) \flba \{\, (\ub{i}, \ub{j}) \flba \ub{j}\we \ub{u}+(-1)^p \, \ub{i} \we \ub{v} \,\}.
\end{equation}
Therefore the dual AK complex is isomorphic to
\[
Q\,:\,\xymatrix@C=25pt{0\ar[r]&\bwi{r}{\!A}{B}\ar[r]&\bwi{r-1}{\!A}{B}\ar[r]&\ \cdots\cdots\ \ar[r]&\bwi{2}{\!A}{B}\ar[r]&B\ar[r]&A\ar[r]&0}
\]
where $A$ is in degree zero, and the differential is $-(p+1)\, d_{r-p}\he$ on each $\bwi{r-p}{\!A}{B}$.
\par \medskip
We have a natural  quasi-isomorphism $\smash{\xymatrix@C=17pt{\thr\he\ar[r]^-{\sim}&Q.}}$ given by the map $d_{r+1}\he$. Besides, as a complex of $A$-modules, $Q$ splits as the direct sum of $\theta _{I}\he[r]$ and of a null-homotopic complex.
\par \medskip
The isomorphism (\ref{dualityfirst}) induces another one, namely:
\begin{equation}\label{dualitysecond}
\bwi{r-p}{\!A}{I}\simeq \textrm{Hom}_{C}\he (\bwi{p+1}{\!A}{B},\theta _{I}\he).
\end{equation}
\par \medskip
There is a natural product $\apl{\, \wh{*} \, }{P\oti_{B}\he Q}{Q}$ which is defined by the same formula as the product $\,*\,$:
\begin{equation}\label{EqOnzeBis}
 \xymatrix@C=40pt@R=2pt{
\wh{*}\,:\,\bwi{l+1}{A}{B}\oti_{B}\he\bwi{r-k}{A}{B}\ar[r]&\bwi{r-k+l}{A}{B}\hspace*{118pt}\\
\hspace*{10pt}(\ubi_{1}\he,\ubj_{1}\he)\oti(\ubi_{2}\he, \ubj_{2}\he)\ar@{|->}[r]& \be \, (\ubi_{1}\he\we\ubj_{2}\he+(-1)^{l}\be\ubj_{1}\he\we\ubi_{2}\he,\ubj_{1}
\he\we\ubj_{2}\he)
}
\end{equation}
A straightforward computation shows that $\wh{*}$ is indeed a morphism of complexes.
\begin{proposition}[Dual HKR isomorphism, local case]\label{PropDeuxHKR}
Under the hypotheses of Proposition \emph{\ref{PropUnHKR}}, the quasi-isomorphism $\smash{\xymatrix@C=17pt{\theta _{I}\he[r]\ar[r]^-{\sim}&Q}} $ induces isomorphisms in $D^{\emph{b}}\be (A)\!:$
\[
\xymatrix@C=30pt{
\emph{RHom}^{r}_{C}
(A,\thr)\ar[r]^-{\sim}&\emph{RHom}^{r}_{C}(A,Q)&\emph{Hom}_{C}(A,Q)\simeq\bop_{i=0}^{r}\bwi{i}{\!A}{I}[i]\ar[l]_-{\sim}}
\]
\end{proposition}
\begin{proof}	
Since $P$ is a complex of free $A$-modules, the natural map from $\textrm{Hom}_{A}\he(P,\thr)$ to $\textrm{RHom}_{A}\he(P,\thr)$ is an isomorphism. Let us consider the following commutative diagram in $D^{\textrm{b}}(C)$:
\par\smallskip
\[
{\xymatrix@C=25pt{\st
\textrm{RHom}^{\ell}_{C}(A, \,\thr)\ar[d]^-{\sim}\ar[r]^-{\sim}&\st\textrm{RHom}^{\ell}_{C}(P,\,\thr)\ar[d]^-{\sim}&\st \textrm{Hom}_{C}(P,\,\thr)\ar[l]_-{\phi _{1}\he}\ar[d]^-{\sim}\\
\st\textrm{RHom}_{A}(A\,\lltensr{}_{C}\he A,\,\thr)\ar[r]^-{\sim}&\st\textrm{RHom}_{A}(A\,\lltensr{}_{C}\he P,\,\thr)&\st \textrm{Hom}_{A}\he(A\,\oti_{C}\he P,\,\thr)\ar[l]_-{\phi _{2}\he}\\
\st\textrm{RHom}^{r}_{C}(A,\,\textrm{RHom}_{A}(A,\,\thr))\ar[u]_-{\sim}\ar[r]^-{\sim}&\st\textrm{RHom}^{r}_{C}(A,\,\textrm{RHom}_{A}(P,\,\thr))\ar[u]_-{\sim}&\st \textrm{Hom}_{C}\he(A,\,\textrm{Hom}_{A}(P,\,\thr))\ar[u]_-{\sim}\ar[l]\\
\st\textrm{RHom}^{r}_{C}(A,\,\thr)\ar[u]_-{\sim}\ar[r]^-{\sim}&\st\textrm{RHom}^{r}_{C}(A,\,Q)\ar[u]_-{\sim}&\st \textrm{Hom}_{C}(A,\,Q)\ar[u]_-{\sim} \ar[l]_-{\phi _{3}\he}}
}
\]
\par\bigskip
By Proposition \ref{PropUnHKR}, $\phi _{1}\he$ and $\phi _{2}\he$ are isomorphisms. This implies that $\phi _{3}\he$ is also an isomorphism.
\end{proof}
We provide now another proof of Proposition \ref{PropDeuxHKR}, which gives a more precise result:
\begin{proposition}\label{PropDeuxBisHKR}
Under the hypotheses of Proposition \emph{\ref{PropUnHKR}}, let $\phi $ be the morphism in $D^{\emph{b}}\be(C)$ obtained by the composition
\[
\xymatrix{
\bop_{i=1}^{r}\bwi{i}{A}I[i]\simeq\emph{Hom}_{C}\he(A,Q)\ar[r]&\emph{RHom}_{C}\he(A,\theta  _{I}\he[r])&\emph{Hom}_{C}\he(P,\theta  _{I}\he[r])\simeq\bop_{i=0}^{r}\bwi{i}{A}{I[i]}\ar[l]_-{\sim}
}
\]
where the last isomorphism is \emph{(\ref{dualitysecond})}. Then, as a morphism in $D^{\emph{b}}\be(k)$, $\phi $ acts by multiplication by the sign $(-1)^{\frac{(r-i)(r-i-1)}{2}}\be$ on each factor $\bwi{i}{A}{I[i]}$.
\end{proposition}
\begin{proof}
 Let $L$ be the Koszul complex associated with $(j_{1}\he,\dots, j_{r}\he)$ and $\apl{\gamma }{L}{P}$ be the quasi-isomorphism constructed in the proof of Proposition \ref{PropUnHKR}. We must describe the composition
\[
\xymatrix@R=3pt{
\textrm{Hom}_{C}\he(A,Q)\ar[r]&\textrm{Hom}_{C}\he(L,Q)&\textrm{Hom}_{C}\he(L,\theta  _{I}\he[r])\ar[l]_-{\sim}&\textrm{Hom}_{C}\he(P,\theta  _{I}\he[r])\ar[l]^-{\circ\,\gamma} _-{\sim}\\
}
\]
Let $M$ denote the free $B$-module $I\oti_{A}\he B$ and $\apl{\tau }{M}{B}$ be the \mbox{$B$-linear} form defined by the composition $\smash{\xymatrix@C=17pt{I\oti_{A}\he B\ar[r]&I\oti_{A}\he A=I\ar[r]&B.}} $ If we identify $M$ with $B^{r}\be$ via the basis $(\ba{j}_{1}\he,\dots, \ba{j}_{r}\he)$, the linear form $\tau $ is simply the composition
$\smash[b]{\xymatrix@C=40pt{B^{r}\be\ar[r]^-{(\ba{j}_{1}\he,\dots, \,\ba{j}_{r}\he)}&B.}} $ Thus, using the notation of \S~\ref{cup}, $L\oti_{C}\he B$ is isomorphic to the complex $L(M,\tau )$. We denote this latter complex by $(\ti{L},\delta )$.
\par \bigskip
Hence we get by (\ref{EqTroisPrelim}) the chain of isomorphisms
\[
\textrm{Hom}_{C}\he(L,Q)\simeq\textrm{Hom}_{B}\he(\ti{L},Q)\simeq Q \oti_{B}\he \ti{L}\ee\be \simeq Q \oti_{B}\he (\ti{L}, -\delta )\oti_{A}\he \theta  \ee_{I}[-r] \simeq \theta  \ee_{I}[-r] \oti_{A}\he (\ti{L}, -\delta )  \oti_{B}\he Q.
\]
\par \bigskip
Let $(N,s',s'')$ denote the double complex $(\ti{L}, -\delta )  \oti_{B}\he Q$ and let $s$ be the total differential. To avoid cumbersome notation, we use homological grading for $N$.
\par \medskip
Then, for $0\le p,q\le r$, we have (for the definition of $W_{p,q}\he$, see \S\ \ref{cup}):
\par\bigskip
--\quad $N_{p,\,q}\he=\bwi{p}{A}{M}\oti_{B}\he\bwi{q}{A}{B}\simeq\bwi{p}{A}{I}\oti_{A}\he\bwi{q}{A}{B}$
\par\smallskip
--\quad\!\!$
\xymatrix@C=10pt{s'_{p,\,q}\,:\,\bwi{p}{A}{I}\oti_{A}\he\bwi{q}{A}{B}\ar[r]&\bwi{p}{A}{I}\oti_{A}\he\bwi{q-1}{A}{I}\ar[rrrr]^-{-p\,W_{p-1,\,1}\he \oti \,\id}&&&&\bwi{p-1}{A}{I}\oti_{A}\he I\oti_{A}\he\bwi{q-1}{A}{I}}\\
\xymatrix@=40pt{&&&&\hspace*{30pt}\ar[r]^-{\id\oti\,\we}&\bwi{p-1}{A}{I}\oti_{A }\he\bwi{q}{A}{I}\ar[r]&\bwi{p-1}{A}{I}\oti_{A}\he\bwi{q}{A}{B }}
$
\par
--\quad $s''_{p,\,q}=\id\oti\, [-(r-q+1)\,d_{q}\he]$.
\par\bigskip
Besides, an easy verification yields that for $0\le i\le r$,
\par\medskip
--\quad The morphism \!\!$\xymatrix@C=17pt{\alpha _{i}\he :\bwi{i}{A}{I[i]}\ar[r]&\textrm{Hom}_{C}\he(P,\theta _{I}\he[r] ) \ar[r]&\textrm{Hom}_{C}\he(L,Q)\simeq\theta  \ee_{I}[-r]\oti_{A}\he N}$\! is given by the inclusion
\[
\xymatrix@C=30pt{\bwi{i}{A}{I} \ar[r]^-{(-1)^{r}}& \bwi{i}{A}{I}\simeq\theta  \ee_{I}\oti_{A}\he(\bwi{i}{A}{I}\oti_{A}\he\bwi{r}{A}{I})\ar[r]&\theta  \ee_{I}\oti_{A}\he(\bwi{i}{A}{I}\oti_{A}\he\bwi{r}{A}{B})=\theta  \ee_{I}\oti N_{i,\,r}\he}
\]
\par\smallskip
--\quad The morphism \!\!$\xymatrix@C=17pt{\beta _{i}\he:\bwi{i}{A}{I[i]}\ar[r]&\textrm{Hom}_{C}\he(A,Q)\ar[r]&\textrm{Hom}_{C}\he(L,Q)\simeq \theta  \ee_{I}[-r]\oti_{A}\he\! N}$\! is given by the inclusion
\[
\xymatrix@C=30pt{
\bwi{i}{A}{I} \ar[r]^-{(-1)^{r} \be}&\bwi{i}{A}{I}\simeq\theta  \ee_{I}\oti_{A}\he(\bwi{r}{A}{I}\oti_{A}\he\bwi{i}{A}{I})\ar[r]&\theta  \ee_{I}\oti_{A}\he(\bwi{r}{A}{B}\oti_{A}\he\bwi{i}{A}{I})=\theta  \ee_{I}\oti N_{r,\,i}\he
}
\]
\par\medskip
Let $R$ be the subcomplex of $N$ defined by $R_{p,\,q}\he=\bwi{p}{A}{I}\oti_{A}\he\bwi{q}{A}{I}\suq N_{p,\,q}\he.$
\par\medskip
\emph{Claim 1.} For $r\le n\le 2r$, $\ker s_{n}\he=R_{n}\he$.
\begin{proof}
For any $x$ in $\ker s_{n}\he$, let $x_{p,\,q}\he$ (with $n-r\le p,q\le r$ and $p+q=n$) denote the graded components of $x$. Then we have
\par\medskip
--\quad $s''_{r,\,n-r}(x_{r,\, n-r}\he)=0$
\par\smallskip
--\quad $s'_{i,\,n-i}(x_{i,\,n-i}\he)+(-1)^{i-1}\be s''_{i-1,\,n-i+1}(x_{i-1,\,n-i+1}\he)=0$
\par\smallskip
--\quad $s'_{n-r,\,r}(x_{n-r,\,r}\he)=0$
\par\medskip
Notice that for $p,q\ge 0$, $R_{p,\,q}\he\suq\ker s'_{p,\,q}$. Furthermore, if $q$ is positive, $R_{p,\,q}\he=\ker s''_{p,\,q}$. Thus, if $n>r$, $x_{r,\,n-r}\he$ belongs to $R_{r,\,n-r}\he$ and it follows that $x_{i-1,\,n-i+1}\he$ belongs to $R_{i-1,\,n-i+1}\he$ for $n-r+1\le i\le r$. If $n=r$, $N_{r,\,0}\he=R_{r,\,0}\he$ and $s'_{r,\,0}=s''_{r,\,0}=0$. Thus $x_{r,\,0}\he$ belongs to $R_{r,\,0}\he$ and $s''_{r-1,\,1}(x_{r-1,\,1}\he)=0$. Hence $x_{r-1,\,1}\he$ belongs to $R_{r-1,\,1}\he$ and we argue as in the case $n>r$. This proves the claim.
\end{proof}
For any integers $p$ and $q$ such that $0\le p,q\le r$ and $p+q\ge r$, let $\smash{\apl{\pi _{p,\,q}\he}{R_{p,\,q}\he}{R_{p+q-r,\,r}\he}} $ be defined by the composition
\[
\xymatrix@C=20pt{
\pi _{p,\,q}\he\,:\,\bwi{p}{A}{I}\oti_{A}\he\bwi{q}{A}{I}\ar[rrr]^-{W_{p+q-r,\,r-q}\he\,\oti\,\id}&&&\bwi{p+q-r}{A}{I}\oti_{A}\he\bwi{r-q}{A}{I}\oti_{A}\he\bwi{q}{A}{I}\ar[rr]^-{\id\,\oti\,\we}&&\bwi{p+q-r}{A}{I}\oti_{A}\he\bwi{r}{A}{I.}
}
\]
Then, for any integer $n$ such that $r\le n\le 2r$, we define a projector $\smash{\apl{\pi _{n}\he}{R_{n}\he}{R_{n-r,\,r}\he}} $ by the formula
\[
\pi _{n}\he=\sum_{p=n-r}^{r}\epsilon_{n,\,p}\he\,\binom{p}{n-r}\,\pi _{p,\,n-p}\he.
\]
\par \medskip
where $\epsilon_{n,p}\he=(-1)^{\frac{(p+1)(p+2)}{2}-\frac{(n-r+1)(n-r+2)}{2}}\be$.
\par\bigskip
\emph{Claim 2.} For $n\le r\le 2n$, $\ker \pi _{n}\he=\im s_{n+1}\he$.
\begin{proof}
We begin by proving the inclusion $\im s_{n+1}\he\suq \ker \pi _{n}\he$. The module $\im s_{n+1}\he$ is spanned by elements of the form $s'_{i,\,n+1-i}(x)+(-1)^{i}s''_{i,\,n+1-i}(x)$, with $n+1-r\le i\le r $ and $x$ in $N_{i,\,n+1-i}\he$.
If $y$ denotes the projection of $x$ on $\bwi{i}{A}{I}\oti_{A}\he\bwi{n-i}{A}{I}\vphantom{^{\ds)}}$, we have by (\ref{EqDeuxBisSecAlgExt}) the identity
\[
(\id\oti \,\we\,)\,(W_{n-r,\,r-n+i-1}\he\oti\id)\,(\id\oti \,\we \,)\,(W_{i-1,\,1}\oti\id)(y)=(\id\oti \,\we \,)\,(W_{n-r,\,r-n+i}\he\oti\id)(y).
\]
This implies that
\[\dfrac{1}{i}\,\pi _{i-1,\,n+1-i}\he\,(s'_{i, \, n+1-i}(x))=\dfrac{1}{r-n+i}\,\pi _{i,\,n-i}\he\,(s''_{i,\,n+1-i}(x)).
\]
Hence we get
\begin{align*}
 \pi _{n}\he[s'_{i,\,n+1-i}(x)+(-1)^{i}s''_{i,\,n+1-i}(x)]&=\epsilon_{n,\,i-1}\he\,\binom{i-1}{n-r}\,\pi _{i-1,\,n+1-i}\he(s'_{i,\,n+1-i}(x))\\
&\hspace*{-60pt}+(-1)^{i}\be\epsilon_{n,\,i} \he\,\binom{i}{n-r}\,\pi _{i,\,n-i}\he(s''_{i,\,n+1-i}(x))\\
&\hspace*{-130pt}=\dfrac{\pi _{i}\he(s''_{i,\,n+1-i}(x))}{r-n+i}\tim\left(i\,\epsilon_{n,\,i-1} \he\binom{i-1}{n-r}+(-1)^{i}\be \epsilon_{n,\,i}\he (r-n+i)\binom{i}{n-r}\!\right)=0.
\end{align*}
Since we know that $\alpha _{n-r}\he$ is a quasi-isomorphism in degree $-(n-r)$, by the first claim we have the equality
$R_{n-r,\,r}\he\oplus\im s_{n+1}\he=\ker s_{n}\he=R_{n}\he$. It follows that the kernel of the projector $\pi _{n}\he$ is exactly the image of $s_{n+1}\he$.
\end{proof}
\par\medskip
A quick computation shows that the map $\smash[t]{\apl{\pi _{r,\,n-r}\he}{R_{r,\,n-r}\he}{R_{n-r,\,r}\he}} $ is the multiplication by the constant $\smash{(-1)^{n(n-r)}\be \bigm/\binom{r}{n-r}}$. Thus $\pi _{n\,|\,R_{r,\,n-r}\he}\he=(-1)^{n(n-r)}\be \epsilon_{n,\,r}\tim\id$. It follows from the second claim that
\[
\im\bigl[\alpha _{n-r}\he- (-1)^{n(n-r)}\be \epsilon_{n,\,r} \,\beta _{n-r}\he\bigr]\suq \theta  \ee_{I}\oti\im s_{n+1}\he.
\] This yields the result.
\end{proof}

\subsection{The construction of Arinkin and C\u{a}ld\u{a}raru}\label{SectionArikinCalda}
Let $M$ be the free $B$-module $B\oti_{A}\he I$ and let $\tau$ be the $B$-linear form on $M$ obtained via the composition
$\smash{\sutr{B\oti_{A}\he I}{A\oti_{A}\he I=I}{B.}} $
\begin{definition}\label{DefUnAriKinCalda}
 The \emph{Arinkin--C\u{a}ld\u{a}raru complex} $(K,\nu )$ associated with the pair $(I, A)$ is the tensor algebra $K=\bigoplus\limits_{i\ge 0}\bigo{A}{i}M[i]$ endowed with a differential $\nu$ given for any positive integer $p$ by the formula
\[
\nu _{-p}\he(m_{1}\he\oti\dots\oti m_{p}\he)=\smash[t]{\dfrac{1}{p!}} \, \tau (m_{1}\he)\,m_{2}\he\oti\dots\oti m_{p}\he.
\]
\end{definition}
 Remark that $(K,\nu )$ is a free resolution of $A$ over $B$. Indeed, for any nonnegative integer $p$,
\begin{equation}\label{EqUnArikinCalda}
\xymatrix{\bigo{B}{\,p}M\simeq\bigo{A}{\,p+1}\!I\oplus \bigo{A}{\,p}I}
\end{equation}
and the map $p!\,\nu _{-p}\he$ is simply the composition $\sutr{\bigo{B}{\,p}M}{\bigo{A}{\,p}I}{\bigo{B}{\,p-1}\!M.}$
The $B$-module structure on $\bigo{B}{\,p}M$ is given via the isomorphism (\ref{EqUnArikinCalda}) by
\begin{equation}\label{EqDeuxArikinCalda}
 (a+i)\,.\,(\ubi ,\,\ubj)=(a\,\ubi+i\oti\ubj,\,a\,\ubj).
\end{equation}
Besides, there is a canonical sequence of $B$-modules
\begin{equation}\label{EqTroisArikinCalda}
 \sutrgdpt{\bigo{A}{\,p}I}{\bigo{B}{\,p}M}{\bigo{A}{\,p-1}I}{.}
\end{equation}
Let $\mathfrak{a}$ be the antisymmetrization map from $\bigo{A}{}I$ to $\bwi{}{\!A}{I}$ defined by (\ref{EqUnPrelim}). Then there exists a natural morphism \[\apl{\zeta _{-p}\he\, }{\bigo{B}{\,p}M}{\bwi{p+1}{\!A}{B}}\]
\par \medskip
given by
$\zeta _{-p}\he(\ubi,\ubj)=({\mathfrak{a}}_{\,p+1}\he(\ubi), {\mathfrak{a}}_{\,p}\he(\ubj))$. Thanks to (\ref{EqQuatreDgAlg}) and ($\ref{EqDeuxArikinCalda}$), $\zeta _{-p}\he$ is $B$-linear. Besides, if $P$ is the AK complex associated with $I$, then ${\apl{\zeta }{K}{P}} $ is a morphim of complexes which is a quasi-isomorphism and commutes to the quasi-morphisms $\smash{\xymatrix@C=17pt{K\ar[r]^-{\sim}&A}}$ and $\smash{\xymatrix@C=17pt{P\ar[r]^-{\sim}&A.}}$
\par\medskip
This construction allows to prove Arinkin--C\u{a}ld\u{a}raru's HKR theorem in the local case:
\begin{proposition}[{\cite{AC}}]\label{PropUnArinkinCalda}
 Under the hypotheses of Proposition \emph{\ref{PropUnHKR}}, the map obtained as the composition
\[
\xymatrix@C=40pt{A\,\lltensr{}_{C}\he\,A&A\,\lltensr{}_{C}\he\,K\ar[l]_-{\sim}\ar[r]&A\,\oti_{C}\he\,K \simeq\smash{\bigoplus\limits_{i\ge 0}}
\bigo{A}{i}I[i]\ar[r]^-{\smash[t]{\bop_{i=0}^{r} {\mathfrak{a}_i} \he} }&\smash[t]{\bop_{i=0}^{r}} \bwi{i}{\!A}{I[i]}}
\]
is an isomorphism in $D^{\emph{b}}(A)$.
\end{proposition}
\begin{proof}
 We prove that this morphism is exactly the HKR isomorphism appearing in Proposition \ref{PropUnHKR}. This is done by looking at the commutative diagram:
\[
\xymatrix@C=35pt@R=1pt{
&A\,\lltensr{}_{C}\he\,K\ar[dd]^-{\id\lltensr{}_{C}\he \, \zeta }\ar[dl]_-{\sim}\ar[r]&A\oti_{C}\he K\ar[dd]^-{\id\oti_{C}\, \zeta }\ar@<1ex>@{}[r]_-{\ds \simeq}&\bop_{i\ge 0}\bigo{A}{i}I[i]\ar[dd]^-{\bop_{i=0}^{r}{\mathfrak{a}_i} \he}\\
A\,\lltensr{}_{C}\he\,A&&\\
&A\,\lltensr{}_{C}\he\,P\ar[ul]^-{\sim}\ar[r]^-{\sim}&A\oti_{C} \he P\ar@<1ex>@{}[r]_-{\ds \simeq}&\bop_{i=0}^{r}\bwi{i}{\!A}I[i]
}
\]
\end{proof}
\subsection{Additional properties of local Atiyah complexes}\label{SousSecAKComplexes}
Let  $\Omega _{A/k}\he$ be the module of K\"{a}hler diffe\-ren\-tials of $A$ over $k$, and put $\om{i}=\bwi{i}{A}{}\,\Omega _{A/k}\he$ for $1\le i\le r$. An $A$-connection $\nabla$ on $I$
is a $k$-linear morphism $\smash{\apl{\nabla \, }{\, I}{\,\Omega _{A/k}\he\oti_{A}\he I}}$ satisfying Leibniz's rule $\nabla(a\,i)=a\nabla i+da\oti i$ for any $a$ in $A$ and any $i$ in $I$. In our setting, an $A$-connection
on $I$ is the same thing as the datum of a $k$-vector space of rank $r$ in $I$ (corresponding to the space of flat sections of $\nabla$).
\par \medskip
Recall that the automorphism group of $B$ in the category of $k$-extensions of $A$ by $I$ is the set $\textrm{Der}_{k}\he(A,I)$ of $k$-derivations of $A$ with values in $I$, which is isomorphic to $\smash[b]{\textrm{Hom}_{A}\he(\om{},I)} $. For such a derivation $\chi $, we denote by $u_{\chi }\he$ the associated isomorphism of $B$ given explicitly by the formula $u_{\chi }\he(i,a)=(i+\chi (a),a)$.
\begin{proposition}\label{PropTroisHKR}
 Let $\chi $ be an element of $\emph{Der}_{k}\he(A,I)$ and $\wh{\chi }$ be the associated morphism in $\emph{Hom}_{A}\he(\om{}, I)$. Then:
 \medskip
\begin{enumerate}
 \item [(1)] Every $A$-connection on $I$ induces a $u_{\chi }\he$-linear isomorphism of the AK-complex $P$ (resp.\ of the dual AK complex $Q$) commuting with the quasi-isomorphism $\smash[t]{\xymatrix@C=17pt{P\ar[r]^-{\sim}&A}} $ (resp.\ $\smash[t]{\xymatrix@C=17pt{\theta_{I}\he[r] \ar[r]^-{\sim}&Q\!\!}} $).
\par \smallskip
\item[(2)] If $\smash[b]{\apl{\, \wh{\chi \,}}{\, \om{}}{I}}$ is an isomorphism, there exists a canonical $u_{\chi }\he$-linear isomorphism of $P$ (resp.\ $Q$) commuting with the quasi-isomorphism $\smash{\xymatrix@C=17pt{P\ar[r]^-{\sim}&A}}$  (resp.\ $\smash{\xymatrix@C=17pt{\theta_{I}\he[r] \ar[r]^-{\sim}&Q\!\!}}
$).
\end{enumerate}
\end{proposition}
\begin{proof}
 (1) For any positive integer $p$, an $A$-connection $\nabla$ on $I$ induces an $A$-connection $\bw{p}{}{\,\nabla}$ on $\bwi{p}{\!A}{I}$. Let $\apl{R_{p}\he}{\bwi{p}{\!A}{I}}{\bwi{p+1}{\!A}{I}}^{\vphantom{\ds)}}$ be defined as the composition
\[
\xymatrix@C=30pt{\bwi{p}{\!A}{I}\ar[r]^-{\bwi{p} {\vphantom{A}}{\,\nabla}}&\om{}\oti_{A}\he\bwi{p}{\!A}{I}\ar[r]^-{\chi\, \oti\,\id}&I\oti_{A}\he\bwi{p}{\!A}{I}\ar[r]^-{\wedge{}{}}&\bwi{p+1}{\!A}{I.}}
\]
Using the isomorphism (\ref{EqDeuxDgAlg}), we define $\apl{\varphi _{-p}\he}{\bwi{p+1}{\!A}{B}}{\bwi{p+1}{\!A}{B}}$ by $\varphi _{-p}\he(\ubi ,\ubj  )=(\ubi +R_{p}\he(\ubj ),\ubj ).$
\par \smallskip
Then, using (\ref{EqQuatreDgAlg}), we obtain that for any $(i,a)$ in $B$,
\begin{align*}
 \varphi _{-p}\he
[(i,a)*(\ubi ,\ubj )]&=\varphi _{-p}\he(a\ubi +i\wedge\ubj ,a\ubj )\\
&=(a\ubi +i\wedge\ubj +R_{p}\he(a\ubj ),a\ubj )\\
&=(a\ubi +i\wedge\ubj +aR_{p}\he(\ubj )+\chi (a)\wedge \ubj, a\ubj )\\
&=\bigl[a (\ubi +R_{p}\he(\ubj ))+(i+\chi (a))\wedge \ubj, a\ubj \bigr]\\
&=u_{\chi }\he(i,a)*(\ubi +R_{p}\he(\ubj ),\ubj )\\
&=u_{\chi }\he(i,a)*\varphi _{-p}\he(\ubi ,\ubj ).
\end{align*}
If we take for $\apl{\varphi _{0}\he}{B}{B}$ the isomorphism $u_{\chi }\he$, which is of course $u_{\chi }\he $-linear, the $\varphi _{-p}\he$'s define the required automorphism of $P$.
\par\medskip
(2) For any positive integer $p$, the map $\wh{\chi }$ induces an isomorphism $\smash{\apl{\bw{p}{}\, \wh{\chi }\, }{\om{p}}{\, \bwi{p}{\!A}{I.}}}$ Then we define ${\apl{R_{p}\he}{\bwi{p}{\!A}{I}}{\bwi{p+1}{\!A}{I}}} $ by ${R_{p}\he =\bwi{p+1}{}{}\, \wh{\chi }\circ \mathfrak{d}_{p}\he\circ(\bwi{p}{}{}\, \wh{\chi }\, )^{-1}}\be$, where ${\apl{\mathfrak{d}_{p}\he}{\om{p}}{\om{p+1}}}$ is the exterior differential. For $a$ in $A$ and $\ubj $ in $\bwi{p}{\!A}{I}$, we have
\begin{align*}
 R_{p}\he(a\ubj )&=\bw{p+1}{}{}\, \wh{\chi }\,\bigl[\, a\mathfrak{d}_{p}\he\bigl( (\bwi{p}{}{}\, \wh{\chi }\, )^{-1}\be(\ubj )\bigr)+da\wedge(\bw{p}{}{}\, \wh{\chi }\, )^{-1}\be(\ubj )\bigr]\\
&=aR_{p}\he(\ubj )+\wh{\chi }(da)\wedge\ubj\\
&=aR_{p}\he(\ubj )+\chi (a)
\wedge\ubj\!.
\end{align*}
Then we argue exactly as in (i).
\end{proof}
This proposition implies as a corollary that the two local HKR isomorphisms of Proposition \ref{PropUnHKR} and Proposition \ref{PropDeuxHKR} are in fact independent of the splitting of the Atiyah sequence (\ref{EqUnHKR}), since two different splittings yield isomorphic extensions.
\section{Atiyah complexes (II)}\label{Atiyah Global}
In this section, we fix two connected analytic manifolds $X$ and $Y$ such that $X$ is a proper closed complex submanifold of $Y$\!. We introduce some notation which will be used extensively in the sequel: $r$ is the codimension of $X$ in $Y$\!, $\smash{\apl{j}{X}{Y}}$ is the canonical inclusion, $\smash[t]{\bax}$ is the first formal \nbh\ of $X$ in $Y$\!, $\smash{\apl{\ba{j}}{X}{\bax }}$ is the associated inclusion and $\bb$ is the trivial $\C_{X}\he$-extension of ${\oox}^{\vphantom{stupide}}$ by $N\ee_{X/Y}$. Remark that by the adjunction formula, $\smash{\textrm{det} \, N_{X/Y}\he[-r]}$ is isomorphic to the relative dualizing complex $\smash{\omega_{X/Y}\he}$.
Although $\omega_{X/Y} \he $ is an object of $D^{\textrm{b}}\be(X)$, we will always consider it as the object $\det N_{X/Y}\he[-r]$ in the category of complexes of sheaves of $\oox$-modules.
The Atiyah sequence associated with the pair $(X,Y)$ is the exact sequence
\begin{equation}\label{EqUnAnalyticHKR}
 \sutrgd{N\ee_{X/Y}}{\oo\baxx}{\oox }
\end{equation}
in $\textrm{Mod}(\oo\baxx)$, which is a sheafified version of (\ref{EqUnHKR}).
\par\medskip
\begin{definition}
A \emph{quantized analytic cycle} in a complex manifold $Y$ is a couple $(X, \sigma)$ such that:
\begin{enumerate}
\item[--] $X$ is a closed complex submanifold of $Y.$
\item[--] $\sigma$ is a holomorphic retraction of $\ba{j}.$
\end{enumerate}
\end{definition}
\par \smallskip
If $(X, \sigma)$ is a quantized analytic cycle, then the Atiyah sequence (\ref{EqUnAnalyticHKR}) is automatically split over $\oox$.
\subsection{Analytic HKR isomorphisms}\label{anhkr}
The constructions of \S~\ref{SectAlgExt} can be sheafified  in an obvious manner. Thus, for every positive integer $p$, $\bwi{p}{\oox}{\bb}_{\vphantom{\ds)}}$ is naturally a sheaf of $\bb$-modules on $X$. We get in this way two AK complexes $\pp$ and $\qq$ which are complexes of $\bb$-modules.
\par\medskip
If $\apl{\sigma }{\bax}{X}$ is the retraction of $\ba{j}$ obtained from a splitting of (\ref{EqUnAnalyticHKR}), then $\sigma $ induces an isomorphism $\smash{\aplexp{\psi _{\sigma }}{\bb}{\oo\baxx}{\sim}}$of $\C_{X}\he$-algebras.
\begin{definition}\label{DefUnAnalyticHKR}
 For any positive integer $p$, we put $\bwi{p}{\sigma \he}{\,\oo\baxx}=\psi \ee_{\sigma }\bigl( \bwi{p}{\oox }{\bb}\bigr)$ in $\textrm{Mod}(\oo\baxx)$, and we define the AK complexes $\pp_{\sigma }\he$ and $\qq_{\sigma }\he$ by $\pp_{\sigma }\he=\psi \ee_{\sigma }\pp$ and $\qq_{\sigma }\he=\psi \ee_{\sigma }\qq \vphantom{\Bigl)}$. They are both complexes of $\oo\baxx$ -modules.
\end{definition}
The results of \S~\ref{SousSecAKComplexes} can be extended in our setting.
\begin{proposition}[HKR isomorphism, global case]\label{PropUnAnalyticHKR}
Let $(X, \sigma)$ be a quantized analytic cycle of codimension $r$ in a complex manifold $Y$\!. Then for any locally free sheaves $\mathcal{E}$ and $\eee'$ on $X$, the quasi-isomorphism $\smash{\xymatrix@C=17pt{\pp_{\sigma }\he \oti_{\oox} \he \mathcal{E}' \ar[r]^-{\sim}& \mathcal{E}'}}$ in $\emph{Mod}\,(\oo\baxx)$ induces isomorphisms in $D^{\emph{b}}\be(\oox)$:
\begin{align*}
 &\xymatrix@C=15pt{\Gamma _{\sigma }\he\,:\,\eee \lltensr{}_{\ooy }\he \mathcal{E}' &\eee \lltensr{}_{\ooy }\he (\pp_{\sigma }\he \oti_{\oox} \he \mathcal{E}')  \ar[l]_-{\sim}\ar[r]^-{\sim}&\eee \oti_{\oo_{Y}}\he (\pp_{\sigma }\he \oti_{\oox} \he \mathcal{E}')\,\simeq\,\bop_{i=0}^{r}\eee\, \oti \,\mathcal{E}' \,\oti\, \bwi{i}{}{N\ee_{X/Y}}\,[i]}\\
&\xymatrix@C=15pt@R=0pt{{\Gamma }\ee_{\sigma }\,:\,\rhl{\ooy }
(\eee ,\,\eee' )\ar[r]^-{\sim}&\rhl{\ooy } (\pp_{\sigma }\he \oti_{\oox}\he {\eee},\,\eee' )&\hoo_{\ooy }(\pp_{\sigma }\he \oti_{\oox}\he {\eee},\,\eee' ) \ar[l]_-{\sim}\\
&&\simeq\, \bop_{i=0}^{r} \hoo (\eee,\, \eee') \, \oti \bwi{i}{}{N_{X/Y}\he} \, [-i].}
\end{align*}
\end{proposition}
\begin{proof}
We refer the reader to the proof of Proposition \ref{PropUnHKR}.
\end{proof}
\begin{remark}
For any locally-free sheaves $\eee$ and $\eee'$ on $X$, there are canonical isomorphisms
\[
\eee \, \lltensr{}_{\ooy} \eee' \simeq \eee \, \lltens{}_{\oox} \he  ( \oox \, \lltensr{}_{\ooy} \he \eee' )\quad\textrm{and}\quad \rhl{\ooy }(\eee ,\eee' ) \simeq \eee' \,  \lltens{}_{\oox} \he \rhl {\ooy}(\eee, \oox).
\]
in $D^{\textrm{b}}\be(\oox)$ which are compatible with the HKR isomorphisms of Proposition \ref{PropUnAnalyticHKR}.
\end{remark}
\par \medskip
As in the local case, we also have a dual HKR isomorphism. To state the result, we consider for any holomorphic vector bundle $\eee$ on $X$ the isomorphism
\begin{align} \label{right}
\hoo_{\ooy }\he(\eee , \qq_{\sigma }\he\, \oti_{\oox} \he  \omega_{X/Y} \he  \oti_{\oox} \he \eee')\simeq\bop_{i=0}^{r}\hoo (\eee,\, \eee')  \, \oti \, \bwi{i}{}{N_{X/Y}\he} \, [-i]
\end{align}
given by the \textit{left} duality map $D^{\ell}\be$ introduced in \S~\ref{cup} (for this we consider the normal bundle as the dual of the conormal bundle, so that (\ref{right}) is an isomorphism of left modules over the graded exterior algebra of $N\ee_{X/Y}$).
\begin{proposition}[Dual HKR isomorphism, global case]\label{HKR2}
Let $(X, \sigma)$ be a quantized analytic cycle of codimension $r$ in a complex manifold $Y$\!. Then for any locally free sheaves $\mathcal{E}$ and $\eee'$ on $X$, the quasi-isomorphism $\smash{\xymatrix@C=17pt{\eee' \ar[r]^-{\sim}&\qq_{\sigma }\he\, \oti_{\oox} \he  \omega_{X/Y} \he \oti_{\oox} \he \eee'}}$ induces an isomorphism
\begin{align*}
&\xymatrix@C=35pt@R=1pt{\wh\Gamma_{\sigma }\, :\, \rhr{\oo_{Y}}(\eee ,\eee')\ar[r]^-{\sim}&\rhr{\oo_{Y}}(\eee ,\, \qq_{\sigma }\he\, \oti_{\oox} \he  \omega_{X/Y} \he  \oti_{\oox} \he \eee')\hspace*{20pt}}\\[-1ex]
&\xymatrix@C=35pt@R=1pt {\hspace*{40pt}&\hoo_{\ooy }\he(\eee ,\, \qq_{\sigma }\he\, \oti_{\oox} \he  \omega_{X/Y} \he\oti_{\oox} \he \eee')\ar[l]_-{\sim} \simeq \, \bop_{i=0}^{r}\hoo (\eee,\, \eee') \,  \oti \, \bwi{i}{}{N_{X/Y}\he} \, [-i] \ar[l]_-{\sim}}
\end{align*}
 in $D^{\emph{b}}(\oox)$, the last isomorphism being given by \emph{(\ref{right})}.
\end{proposition}
\begin{proof}
We refer the reader to the proof of Proposition \ref{PropDeuxHKR}.
\end{proof}
\par \medskip
The set of retractions of $\ba{j}$ is an affine space over $\smash[b]{\textrm{Der}_{\C_{X}}\he(\oox , N\ee_{X/Y})} $, the latter being isomorphic to $\smash[b]{\textrm{Hom}_{\oo_{X}}\he(\Omega ^{1}_{X}, N\ee_{X/Y})} $. The main difference with the local situation is that the HKR isomorphism can depend a priori on $\sigma $. This problem will be discussed in \S~\ref{comparison}. At this stage, we only give the following result, which is the global analog of Proposition \ref{PropTroisHKR}:
\begin{theorem}\label{PropDeuxAnalyticHKR} Let $(X,Y)$ be a pair of complex manifolds such that $X$ is a closed submanifold of $Y$\!.
\begin{enumerate}
 \item [(1)] Assume that $\smash[b]{N\ee_{X/Y}} $ carries a global holomorphic connection. Then for any retractions $\sigma $ and $\sigma '$ of $\ba{j}$, $\pp_{\sigma }\he$ (resp.\ $\qq_{\sigma }\he$) is naturally isomorphic to $\pp_{\sigma '}\he$ (resp.\ $\qq_{\sigma '}\he$) and this isomorphism commutes with the quasi-isomorphism $\smash{\xymatrix@C=17pt{\pp_{\sigma }\he\ar[r]^-{\sim}&\oox }}$ (resp.\ $\smash{\xymatrix@C=17pt{\omega_{X/Y}^{\otimes -1}\ar[r]^-{\sim}&\qq_{\sigma }\he}}\!\!$).
 \par \smallskip
\item[(2)] Let $\sigma $ and $\sigma '$ be two retractions of $\ba{j}$ such that the element $\sigma '-\sigma $ in $\emph{Hom}_{\oox }\he(\Omega ^{1}_{X}, N\ee_{X/Y})$ is an isomorphism. Then $\pp_{\sigma }\he$ (resp.\ $\qq_{\sigma }\he$) is naturally isomorphic to $\pp_{\sigma '}\he$ (resp.\ $\qq_{\sigma '}\he$) and this isomorphism commutes with the quasi-isomorphism $\smash{\xymatrix@C=17pt{\pp_{\sigma }\he\ar[r]^-{\sim}&\oox }}$(resp.\ $\smash{\xymatrix@C=17pt{\omega_{X/Y}^{\otimes -1}\ar[r]^-{\sim}&\qq_{\sigma }\he}}\!\!$).
\end{enumerate}

\end{theorem}
\begin{proof}
We refer the reader to the proof of Proposition \ref{PropTroisHKR}.
\end{proof}
As a consequence, we obtain immediately:
\begin{corollary}\label{ThUn}
 Assume that $Y=X\tim X$, and let $\sigma _{1}\he$ and $\sigma _{2}\he$ be the retractions of $\ba{j}$ induced by the first and second projections. For any complex number $t$, we put $\sigma_{t}\he=(2-t)\sigma _{1}\he+(t-1)\sigma _{2}\he$. Then for any $s$ and $t$ in $\C$, $\Gamma _{\sigma _{s}\he}\he=\Gamma _{\sigma _{t}\he}\he$.
\end{corollary}
\begin{proof}
 The map $\sigma _{1}\he-\sigma _{2}\he$ in $\textrm{Der}_{\C_{X}\he}\he(\oox ,N\ee_{X/X \times X})$ is given by
\[
(\sigma _{1}\he-\sigma _{2}\he)(f)=\{\xymatrix@C=17pt{\! \!(x,y)\ar@{|->}[r]&f(x)-f(y)\! \!}\}\ \textrm{modulo}\ \jj_{X}^{2}.
\]
It induces an isomorphism between $\Omega ^{1}_{X}$ and $N\ee_{X/X \times X}$. Now, for any complex numbers $s$ and $t$ such that $s\neq t$, $\sigma _{t}\he-\sigma _{s}\he=(s-t)(\sigma _{1}\he-\sigma _{2}\he)$ so that $\Gamma _{\sigma _{s}\he}\he=\Gamma _{\sigma _{t}\he}\he$ by Theorem \ref{PropDeuxAnalyticHKR} (2).
\end{proof}
\subsection{The twisted case}\label{TwistedCase}
For any nonnegative integer $p$, the sheaf $\bwi{p+1}{\sigma \he}{\oo\baxx}$ is isomorphic (as a sheaf of $\oox $-modules) to $^{\vphantom{\bigl)}}\bwi{p+1}{}{N\ee_{X/Y}}\oplus\bwi{p}{}{N\ee_{X/Y}}$ via the isomorphism (\ref{EqDeuxDgAlg}). Besides, any section $s$ of the sheaf $\hoo^{\vphantom{\bigl)}}_{\oox}(\bwi{p}{}{N\ee_{X/Y}},\bwi{p+1}{}{N\ee_{X/Y}})$ induces a section of $\mathcal{A}ut_{\oo\baxx\, }(\bwi{p+1}{\sigma \he}{\oo\baxx}) {}^{\vphantom{a}}$ given by
\[
\xymatrix@C=17pt{(\ubi ,\ubj )\ar@{|->}[r]&(\ubi +s(\ubj ),\ubj )}
\]
so that we have a cano\-ni\-cal embedding of ${\hoo_{\oox}\he(\bwi{p}{}{N\ee_{X/Y}},\bwi{p+1}{}{N\ee_{X/Y}})}$ in $\mathcal{A}ut^{\vphantom{stupide}}_{\oo\baxx\, }(\bwi{p+1}{\sigma \he}{\oo\baxx}).$

\par \medskip
Recall that for any vector bundles $\mathcal{E}$ and $\mathcal{F}$ on $X$ and every nonnegative integer $i$, there is a canonical isomorphism between $\textrm{Ext}^{i}_{\oox} (\mathcal{E}, \mathcal{F})$ and $\textrm{H}^{i}\be(X,\hoo_{\oox}\he(\mathcal{E}, \mathcal{F})).$

\begin{definition}\label{DefUnTwisted}
 For any nonnegative integer $p$ and any $\lambda $ in $\smash[b]{\textrm{Ext}^{1}_{\oox}(\bwi{p}{}{N\ee_{X/Y}},\bwi{p+1}{}{N\ee_{X/Y}}))} $, we denote by $\smash{\bwi{p}{\sigma ,\,\lambda }\oo\baxx} $ any twisted sheaf associated with the image of the class $\lambda $ in the \v{C}${\vphantom{\Bigl)}}$ech cohomology group $\check{\textrm{H}}^{1}\be(X,\mathcal{A}ut_{\oo_{\ba{X}}\he}\he(\bwi{p+1}{\sigma \he}{\oo\baxx}))\vphantom{\Bigl)}$.
 \end{definition}

This definition makes sense because all such twisted sheaves are isomorphic. The sheaves $\bwi{p}{\sigma, \,\lambda }{\oo\baxx}$ are sheaves of $\oo\baxx$ -modules which are locally isomorphic to $\bwi{p}{\sigma \he}{\oo\baxx}$. They fit into exact sequences
\begin{equation}\label{EqUnTwisted}
 \sutrgdpt{\bwi{p+1}{}{N\ee_{X/Y}}}{\bwi{p}{\sigma ,\,\lambda }{\oo\baxx}}{\bwi{p}{}{N\ee_{X/Y}}}{.}
\end{equation}
If we fix for each integer $p$ between $0$ and $r-1$ a class $\lambda _{p}$ in
$\textrm{Ext}^{1}_{\oox}(\bwi{p}{}{N\ee_{X/Y}},\bwi{p+1}{}{N\ee_{X/Y}})$, the exact sequences (\ref{EqUnTwisted}) allow to define twisted AK complexes $\pp_{\sigma ,\,\lambda _{0},\dots,\,\lambda _{r-1}}\he$ and $\qq_{\sigma ,\,\lambda _{0},\dots,\,\lambda _{r-1}}^{\vphantom{idiot}} $ which are well-defined modulo isomorphism. Then the results of Proposition \ref{PropUnAnalyticHKR} also hold in the twisted case.
\begin{proposition}\label{PropUnTwisted}
 Let $\sigma $ be a retraction of $\ba{j}$, $\chi $ be in $\emph{Der}_{\C_{X}\he}\he(\oox, N\ee_{X/Y})$ and $\wh{\chi }$ be the associated section of the sheaf $\smash[b]{\hoo_{\oox}\he(\Omega ^{1}_{X}, N\ee_{X/Y})} $. For any nonnegative integer $p$, let $\lambda_{p}$ denote the image of the Atiyah class of $\vphantom{\Bigl)}\bwi{p}{}{N\ee_{X/Y}}$ by $\wh{\chi} \we  \emph{id}$ in
$\emph{Ext}^{1^{\vphantom{idiot}}}_{\oox}(\bwi{p}{}{N\ee_{X/Y}}, \bwi{p+1}{}{N\ee_{X/Y}}).$
Then
$\bwi{p+1}{\sigma +\chi \vphantom{\lambda} }\oo\baxx$ and $\bwi{p+1}{\sigma ,\,\lambda_{p}}{\oo\baxx}$ are isomorphic as sheaves of $\oo\he\baxx$ -modules.
\end{proposition}
\begin{proof}
 Let $\bigl( U_{\alpha }\he\bigr)_{\alpha \in J}$ be an open covering of $X$ such that $N\ee_{X/Y}$ admits a holomorphic connection $\nabla_{\alpha}$ on each $U_{\alpha }\he$. For every $\alpha^{\vphantom{\bigl)}}$ in $J$, $\bwi{p+1}{\sigma +\chi }{\oo\baxx}$ is isomorphic on $U_{\alpha }\he$ to $\bwi{p+1}{\sigma \he}{\oo\baxx}$ via
\[
\xymatrix{\varphi _{\alpha }\he:(\ubi ,\ubj )\ar@{|->}[r]&(\ubi -(\wh{\chi }\wedge\id)\,(\Lambda ^{p}\be\,\nabla\he_{\!\alpha }\,(\ubj )),\ubj ).}
\]
Thus, for any $\alpha $, $\beta $ in $J$, if $M_{\alpha \beta }\he=\Lambda ^{p}\be\,\nabla\he_{\!\alpha\,|\,U_{\alpha \beta } }-\Lambda ^{p}\be\,\nabla\he_{\!\beta\,|\,U_{\alpha \beta } }$,
\[
\varphi _{\beta }\he\circ \varphi_{\alpha }^{-1}(\ubi ,{\ubj })=(\ubi +(\wh{\chi }\wedge\id)\,(M_{\alpha \beta }(\ubj )),\ubj ).
\]
Since $M_{\alpha \beta }\he$ is a $1$-cocycle representing the Atiyah class of the holomorphic vector bundle $\bwi{p}{}{N\ee_{X/Y}}$ in $\textrm{Ext}^{1}_{\oox}(\bwi{p}{}{N\ee_{X/Y}},\Omega ^{1}_{X} {{\vphantom{\Bigl(A}}}\otimes \bwi{p}{}{N\ee_{X/Y}}),$ we get the result.
\end{proof}
We now recall Arinkin--C\u{a}ld\u{a}raru's construction of general analytic HKR isomorphisms and make the link with twisted AK complexes and twisted HKR isomorphisms. Recall that for any locally-free sheaf $\eee$ on $X$, if $\eee$ admits a locally-free extension $\eex$ on $\bax$, there is an exact sequence
\begin{equation}\label{EqDeuxTwisted}
 \sutrgd{N\ee_{X/Y} \oti_{\oox}\he\eee}{\eex}{\eee}
\end{equation}
of sheaves of $\oo\baxx\,$-modules. Thus, for any nonnegative integer $n$, if $\kk_{-n}\he$ is a locally free extension of\ \ $\bigo{}{n} N\ee_{X/Y}$ on $\bax ^{\vphantom {\bigl)}}$, we have an exact sequence
\begin{equation}\label{EqTroisTwisted}
 \xymatrix@C=20pt{0\ar[r]&\bigo{}{n+1}N\ee_{X/Y}\ar[r]^-{i_{n}\he}&\kk_{-n}\he\ar[r]^-{\pi _{n}\he}&\bigo{}{n}N\ee_{X/Y}\ar[r]&0.}
\end{equation}
\begin{definition}\label{defUnTwisted}
If $\smash[b]{\bigl( \kk_{-n}\he\bigr)_{n\ge 0}\he} $ are locally free sheaves on $\bax$ extending $\smash[b]{\bigl( \bigo{}{n}N\ee_{X/Y}\bigr)_{n\ge 0}\he} $, the twisted Arinkin--C\u{a}ld\u{a}raru complex $(\kk,\nu )$ is the complex $\smash[b]{\bop_{n\ge 0}\he\kk_{-n}\he} $ endowed with the differential $\nu$ given for each positive integer $n$ by
$\nu _{-n}\he=\smash[b]{\frac{1}{n!}} \, i_{n-1}\he\circ\pi _{n}\he$.
\end{definition}
\par \medskip
Since the sequences (\ref{EqTroisTwisted}) are exact, $(\kk,\nu )$ is a free resolution of $\oox$ over $\oo\baxx$. The main result about $(\kk,\nu )$ is:
\begin{theorem}[\cite{AC}]\label{PropDeuxTwisted}
Let $(X,Y)$ be a couple of connected complex manifolds such that $X$ is a closed complex submanifold of $Y$ of codimension $r$. Then
\begin{enumerate}
\item[(1)] The complex $\smash[b]{\oox\lltensr{}_{\ooy}\he\oox}$ is formal in $D^{\emph{b}}\be(\oox)$ if and only if $\smash{N\ee_{X/Y}}$ can be exten\-ded to a locally-free sheaf on $\bax {\vphantom{\Bigl )}}$.
\par \smallskip
\item[(2)] If $\bigl( \kk_{-n}\he\bigr)_{n\ge 0}\he$ is a sequence of locally-free sheaves on $\bax$ extending $\bigl(\bigo{}{n}N\ee_{X/Y}\bigr)_{n\ge 0}\he$, then the map
\par
\[
\xymatrix@C=28pt{ \Gamma _{\kk}\he \, : \, \oox \, \lltens{}_{\ooy}\he\oox& \oox\lltens{}_{\ooy}\he\kk\ar[l]_-{\sim}\ar[r]& \oox\oti_{\ooy}\he\kk \, \simeq\, \bop_{i\ge 0}\bigo{}{i}N\ee_{X/Y}\, [i]\ar[r]^-{\smash[t]{\bop \limits_{i=0}^r \mathfrak{a}_i \he}}& \bop_{i=0}^{r}\bwi{i}{}{N\ee_{X/Y}\,[i]}}
\]
is an isomorphism in $D^{\emph{b}}\be(\oox)$.
\end{enumerate}
\end{theorem}
\begin{remark}
$\he$\par
 \begin{enumerate}
  \item [(1)] If the Atiyah sequence (\ref{EqUnAnalyticHKR}) splits, then any retraction $\sigma $ of $\ba{j}$ allows to produce an extension of $N\ee_{X/Y}$ on $\bax$, namely $\sigma \ee\be N\ee_{X/Y}$.
\item[(2)] This theorem appears in \cite{AC} only when $\smash[b]{\kk_{0}\he=\oo\baxx} $ and $\smash[b]{\kk_{-n}\he=\bigo{\oo_{\ba{X}\he}}{n}\kk_{1}\he} $ for $n\ge 1$ (which corresponds to the untwisted case), but the proof remains unchanged $^{\vphantom{\bigl)}}$ under these slightly more general hypotheses.
\end{enumerate}
\end{remark}
Assume now that (\ref{EqUnAnalyticHKR}) splits, and let $\sigma $ be a retraction of $\ba{j}$. Then, if $\smash[b]{\bigl( \kk_{-n}\he\bigr)_{n\ge 0}\he} $ is a sequence of locally-free sheaves on $\bax^{\vphantom{\bigl)}}$ extending $\smash[b]{\bigl(\bigo{}{n}N\ee_{X/Y}\bigr)_{n\ge 0}\he} $, each exact sequence (\ref{EqTroisTwisted}) defines (via $\sigma $) an extension class $\mu _{n}\he$ in ${\textrm{Ext}_{\oox}^{1{\vphantom{\bigl)}}}\bigl(\bigo{}{n} N\ee_{X/Y}, \bigo{}{n+1} N\ee_{X/Y} \bigr)}$.
\par \medskip
\begin{theorem}\label{PropTroisTwisted}
Let $(X, \sigma)$ be a quantized analytic cycle of codimension $r$ in a complex manifold $Y$\!, let $\bigl( \kk_{-n}\he\bigr)_{n\ge 0}\he$ be a sequence of locally-free sheaves on $\bax$ extending $\smash[t]{\bigl(\bigo{}{n}N\ee_{X/Y}\bigr)_{n\ge 0}\he},$ and let $(\mu_{n})_{n \geq 0 } \he$ be the associated extension classes in ${\bigl(}^{\vphantom{stupide}}\emph{Ext}\,_{\oox}^{1}\bigl(\bigo{}{n}N\ee_{X/Y},\bigo{}{n+1 \vphantom{)}}N\ee_{X/Y} \bigr)\bigr)_{n \geq 0}$.
\par \bigskip
For every integer $n$ between $0$ and $r-1$, let $\lambda _{n}$ be any element in
$\emph{Ext}\,_{\oox}^{1} \bigl(\bwi{n}{}{N\ee_{X/Y}},\bwi{n+1}{}{N\ee_{X/Y}} \bigr)$ such that $\mu_{n}{\vphantom{\Bigr(}}\he$ and $\lambda _{n}\he$ map to the same extension class in $\emph{Ext}\,_{\oox}^{1 \vphantom{)}}\bigl(\bigo{}{n}N\ee_{X/Y},\bwi{n+1}{}{N\ee_{X/Y}} \bigr)$ via the antisymmetrization morphisms.
\par \bigskip
Then
$\Gamma _{\kk {\vphantom{\bigl)}}}=\Gamma _{\sigma ,\,\lambda _{0}\he,\dots,\,\lambda _{r-1}\he}\he$.
In particular, if $\kk_{0}\he=\oo\baxx$ and $\kk_{-n}\he=\bigo{\oo_{\ba{X}\he}}{n}\sigma \ee\be N\ee_{X/Y}$ for $n\ge 1$, then $\Gamma _{\kk {\vphantom{\bigl)}}}=\Gamma_{\sigma} \he$.
\end{theorem}
\begin{proof}
 We start with the case $\kk_{0}\he=\oo\baxx$ and $\kk_{-n}\he=\bigo{\oo_{\ba{X}\he}}{n}\sigma \ee\be N\ee_{X/Y}$ for $n\ge 1$, so that all the classes $\mu_n$ vanish. Then for every nonne\-ga\-tive integer $n$, there exists a natural morphism $\smash[b]{\apl{\zeta _{n}\he \, }{\, \kk_{-n}\he}{(\pp_{\sigma})_{-n}\he}} $ given $^{\vphantom{\bigl)}}$by the composition
\[
\xymatrix@C=30pt{
\zeta _{n}\he\, :\, \bigo{\oo_{\ba{X\vphantom{\vert}}}}{n}\sigma \ee\be{N\ee_{X/Y}}\ar[r]&\bwi{n}{\oo_{\ba{X\vphantom{\vert}}\he}}{\sigma \ee\be{N\ee_{X/Y}}}=\oo\baxx\oti_{\oox}\he\bwi{n}{}{{N\ee_{X/Y}}}\ar[r]&\bwi{n+1}{\sigma }{{\oo\baxx}}.
}
\]
where the last arrow is induced by the map $x \fl \sigma^{*}(1) \we x$. Then $\smash[b]{\apl{\zeta }{\kk}{\pp_{\sigma }\he}} $ is a morphism of complexes, which is a global version of the morphism $\zeta $ constructed in  \S~\,\ref{SectionArikinCalda}, and we can argue exactly as in Proposition \ref{PropUnArinkinCalda}.
\par\medskip
In the twisted case, our hypothesis implies that there is a morphism $\smash{\apl{\zeta }{\kk}{\pp_{\sigma ,\,\lambda _{0}\he,\dots,\,\lambda _{r-1}\he}\he}}$ which is locally isomorphic to the previous one. Details are left to the reader.
\end{proof}
\subsection{Comparison of HKR isomorphisms} \label{comparison}
Let $(X, \sigma )$ be a quantized analytic cycle in a complex manifold $Y$\!, and fix two sequences of cohomology classes $\smash[b]{\ub{\lambda \be}=(\lambda _{p}\he)_{0\le p\le r-1}\he} $ and $\smash[b]{\ub{\mu \be}=(\mu _{p}\he)_{0\le p\le r-1}\he} $ such that for each $p$, $\lambda_p$ and $\mu_p$  belong to $\textrm{Ext}\,^{1 \vphantom{\bigl)}}_{\oox}(\bwi{p}{}{N\ee_{X/Y}},\bwi{p+1}{}{N\ee_{X/Y}})$. If $\pp_{\sigma ,\,\ubl } {\vphantom{\Bigl)}} $ and $\pp_{\sigma ,\,\ubm}{\vphantom{\bigl )}}\he$ are the twisted AK
complexes associated $^{\vphantom{\bigl)}}$with $\ubl $ and $\ubm $, the isomorphism
$\xymatrix{\varphi _{\ubm ,\,\ubl }\he:\,\pp_{\sigma ,\,\ubl }\he\ar[r]^-{\sim}&\oox&\pp_{\sigma ,\,\ub{\mu\be}}\he\ar[l]_-{\sim}}$ in $D^{\,\textrm{b}}\be(\ooy)$ induces an automorphism $\Delta_{\sigma} (\ubm ,\, \ubl)$ of $\bop\nolimits_{i=0}^{r}\bwi{i}{}{N\ee_{X/Y}}\,[i]$ in $D^{\,\textrm{b}}\be(\oo_{X}\he)$, as shown in the following diagram:
\[
\xymatrix@C=30pt@C=30pt{
\bop_{j=0}^{r}\bwi{j}{}{{N\ee_{X/Y}}[j]}\ar@<1ex>@{}[r]_-{\ds=}\ar[d]^-{\sim}_-{\Delta_{\sigma} (\ubm ,\, \ubl)}&\oox\oti_{\ooy }\he\pp_{\sigma ,\,\ubl }\he&\oox\,\lltens{}_{\ooy }\he\pp_{\sigma ,\,\ubl }\he\ar[l]_-{\sim}\ar[d]_-{\sim}^-{\id\,\lltens{}_{\ooy }\he\varphi _{\ubm ,\,{\ubl }}}\\
\bop_{i=0}^{r}\bwi{i}{}{{N\ee_{X/Y}}[i]}\ar@<1ex>@{}[r]_-{\ds=}&\oox\oti_{\ooy }\he\pp_{\sigma ,\,\ub{\mu  \be}}\he&\oox\,\lltens{}_{\ooy }\he\pp_{\sigma ,\,\ubm }\he\ar[l]_-{\sim}
}
\]
If we put $\Delta_{\sigma}(\ubl)=\Delta_{\sigma}(0, \ubl)$, then $\Delta_{\sigma}(\ubm, \ubl)=\Delta_{\sigma}(\ubm)^{-1}\be \circ \Delta_{\sigma}(\ubl).$
Recall that
\[
\textrm{Hom}_{D^{\textrm{b}}\be(\oox)}\he\Bigl[ \bop\limits_{j=0}^{r}\bwi{j}{}{{N\ee_{X/Y}}[j]},\,\bop\limits_{i=0}^{r}\bwi{i}{}{{N\ee_{X/Y}}[i]}\Bigr]=\bop_{0\le j \le i\le r}\he \textrm{Ext}^{i-j}_{\oox}\bigl( \bwi{j}{}{{N\ee_{X/Y}}},\bwi{i}{}{{N\ee_{X/Y}}}\bigr).
\]
Therefore, an endomorphism of $\bop_{i=0}^{r}\bwi{i}{}{{N\ee_{X/Y}}[i]}$ in the derived category $D^{\textrm{b}}\be(\oox)$ can be represented by a lower triangular $(r+1)\tim(r+1)$ matrix $\bigl( M_{i,\,j}\he \bigr)_{0\le i,\,j\le r}\he$ such that for $i\ge j$, the entry $M_{i,\,j}\he$ is a cohomology $^{\vphantom{\bigl)}}$class in $ \textrm{Ext}^{i-j}_{\oox}(\bwi{j}{\oox}{{N\ee_{X/Y}}}, \bwi{i}{}{{N\ee_{X/Y}}}).$
\par \bigskip

The computation of the coefficients $\Delta _{\sigma }\he(\ubm,\ubl)_{i, \, j}$ seems to be a delicate problem. We solve it only in particular cases. Let us introduce some preliminary material.
\par\medskip
For any integers $i$ and $j$ such that $0\le j\le i\le r$ and any cohomology class $v$ in $\textrm{H}^{i-j}\be\bigl(X,\bwi{i-j}{}{{N\ee_{X/Y}}}\bigr)$ considered as an element of $\textrm{Hom}_{D^{\textrm{b}}\be(\oox)}\he(\oox, \,\bwi{i-j}{}{{N\ee_{X/Y}}}[i-j])$, we define a morphism $\mathfrak{l}_{i,\,j}\he(v)$ from $\bwi{j}{}{{N\ee_{X/Y}}}[j^{{\vphantom{\bigl)}}}]$ to $\bwi{i}{}{{N\ee_{X/Y}}}[i]$ in $D^{\textrm{b}^{\vphantom{A}}}\be(\oox)$ by the compostion
\[
\xymatrix@C=40pt{\mathfrak{l}_{i,\,j}\he(v) \colon \bwi{j}{}{{N\ee_{X/Y}}}[j] \ar[r]^-{v \, \lltens{}_{\oox} \,\textrm{id}}&\bwi{i-j}{}{{N\ee_{X/Y}}}[i-j] \, \lltens{}_{\oox} \he \bwi{j}{}{{N\ee_{X/Y}}}[j] \ar[r]^-{\wedge}&\bwi{i}{}{{N\ee_{X/Y}}}[i].
}
\]
\par \medskip
In this way, we obtain a morphism
\[
\apl{\, \mathfrak{l}_{i,\,j}\he \, }{\textrm{H}^{i-j}\be\bigl(X,\bwi{i-j}{}{{N\ee_{X/Y}}}\bigr)}{\textrm{Ext}^{i-j}_{\oox}(\bwi{j}{\oox}{{N\ee_{X/Y}}}, \bwi{i}{}{{N\ee_{X/Y}}}).}
\]
\par \medskip
If $\,*\,$ denotes the Yoneda product, for any integers $i,j_{\vphantom{\bigl)}},k$ such that $0 \leq k\le j\le i\le r$ and any cohomology classes $v$ and $w$ in $\textrm{H}^{i-j}\be\bigl( X,\bwi{i-j}{}{{N\ee_{X/Y}}}\bigr)$ and $\textrm{H}^{j-k}\be\bigl( X,\bwi{j-k}{}{{N\ee_{X/Y}}}\bigr)_{\vphantom{stupide}}$ respectively,  we have
\[
\mathfrak{l}_{i,\,j}\he(v)*\mathfrak{l}_{j,\,k}\he(w)=(-1)^{(i-j)(j-k)}\be \,\mathfrak{l}_{i,\,k}\he(v \cup w).
\]
\par \bigskip
We introduce some notation concerning \v{C}ech cohomology. Let $\mathfrak{U}=(U_{\alpha }\he)_{\alpha \in J}\he$ be a locally finite open covering of $X$. For any bounded complex of sheaves $(\ff, d)$ on $X$, we denote by $(\mathscr{C}(\ff),\delta, d )$ the associated \v{C}ech bicomplex, which is quasi-isomorphic to $(\ff, d)$. Besides, we denote by $\we$ the wedge product on the exterior algebra $\bwi{}{}N\ee_{X/Y}$ at the level of \v{C}ech cochains. It is given by the well-known formula:
\[
\apl{\, \we \, }{\,\mathscr{C}^{p}\be\bigl( \bwi{k}{}{{N\ee_{X/Y}}}\bigr)\times\mathscr{C}^{q}\be\bigl( \bwi{l}{}{{N\ee_{X/Y}}}\bigr)}{\mathscr{C}^{p+q}\be\bigl( \bwi{k+l}{}{{N\ee_{X/Y}}}\bigr)}
\]
\[
(\eta \we \eta')_{{\alpha'}_{0}\he,\dots,\,\alpha _{p+q}\he}\he=u_{\alpha_{0}\he,\dots,\,\alpha _{p}\he}\he\we {\eta'}_{\alpha _{p}\he,\dots,\,\alpha _{p+q}\he}\he.
\]
\par \bigskip
Let $v$ be a cohomology class in $\textrm{H}^{i-j}\be\bigl( X,\bwi{i-j}{}{{N\ee_{X/Y}}}\bigr)$. Since $X$ is paracompact, we can choose the covering $\mathfrak{U}$ sufficiently fine in order that $v$ be representable by a \v{C}ech cocycle $(\mathfrak{v}_{\alpha }\he)_{\alpha\, \in \,J^{i-j+1}}\he.$
\par \medskip
Define $\smash[b]{\apl{\, \mathfrak{q}_{i,\,j}\, \he(\mathfrak{v}) \,}{\mathscr{C}\bigl( \bwi{j}{}{{N\ee_{X/Y}}[j]}\bigr)}{{\mathscr{C}\bigl( \bwi{i}{}{{N\ee_{X/Y}}[i]}\bigr)}}} $
by the formula
$\mathfrak{q}_{i,\,j}\he(\mathfrak{v})(\eta )=(-1)^{(i-j)\deg(\eta)}\be\, \mathfrak{v}\we \eta
$, where $\deg(\eta )$ denotes the degree of the \v{C}ech cochain $\eta $. Then $\mathfrak{q}_{i,\,j}\he(\mathfrak{v})$ is a morphism of complexes and the diagram
\[
\xymatrix@C=40pt@R=30pt{
\bwi{j}{}{{N\ee_{X/Y}}[j]}\quad\ar[r]^-{\mathfrak{l}_{i,\,j}\he(v)}\ar[d]_-{\sim}&\bwi{i}{}{{N\ee_{X/Y}}[i]}\quad\ar[d]^-{\sim}\\
\mathscr{C}\bigl( \bwi{j}{}{{N\ee_{X/Y}}[j]}\bigr)\ar[r]^-{\mathfrak{q}_{i,\,j}\he(\mathfrak{v})}&\mathscr{C}\bigl( \bwi{i}{}{{N\ee_{X/Y}}[i]}\bigr)
}
\]
commutes in $D^{\textrm{b}}\be(\oox)$.
\begin{theorem}\label{ThUnCompHKR}
 Let $(X, \sigma)$ be a quantized analytic cycle of codimension $r$ in a complex manifold $Y$\!. Fix two sequences $(c_{n}\he)_{\,0\le n\le r}\he$ and $(d_{n}\he)_{\,0\le n\le r}\he$ of cohomology classes in $\smash[t]{\emph{H}^{1_{\vphantom{)}}}\be\bigl(X, N\ee_{X/Y}\bigr)}^{\he }$.
\par\medskip
 For any integer $p$ between $0$ and $r-1$, let $\lambda _{p}\he$ and $\mu _{p}\he$ be defined in $\emph{Ext}^{1}_{\oox}\bigl(\bwi{p}{}{{N\ee_{X/Y}}},\bwi{p+1}{}{{N\ee_{X/Y}}}\bigr)$ by $\lambda _{p}\he=\mathfrak{l}_{p+1,\,p}\he(c_{p}\he)$ and $\mu _{p}\he=\mathfrak{l}_{p+1,\,p}\he(d_{p}\he)$.
\par\medskip
 Then for any integers $i$ and $j$ such that $0\le j\le i\le r$, we can write $\Delta _{\sigma }\he(\ubm,\ubl)_{i,\,j}\he=\mathfrak{l}_{i,\,j} (\zeta _{i,\,j})\he$, where the classes $\zeta _{i,\,j}\he$ are defined inductively by
\begin{align*}
 \zeta _{i,\,i}\he&=1 &&\textrm{for} \quad 0 \leq i \leq r, \\ \zeta_{i+1,\,0}&=(-1)^i (c_0-d_i)\cup \zeta_{i,\, 0} &&\textrm{for} \quad 0 \leq i \leq r-1,\\[-1ex]
\zeta _{i+1,\,j}\he&=\dfrac{1}{i+1} \bigl[ j\,\zeta _{i,\,j-1}\he + (-1)^{i-j}\be(c_{j}\he-d_{i}\he)\cup\zeta _{i,\,j}\bigr] &&\textrm{for} \quad 1 \leq j \leq i \leq r-1.
\end{align*}
\end{theorem}
\begin{proof}
We choose a locally finite covering $\mathfrak{U}=(U_{\alpha }\he)_{\alpha \,\in\,J}\he$ of $X$ such that for each integer $n$ between $0$ and $r-1$, the classes $c_{n}\he$ and $d_{n}\he$ are representable by \v{C}ech cocycles $(\mathfrak{c}_{n,\,\alpha ,\,\beta }\he)_{\alpha ,\,\beta \,\in\,J}\he$ and $(\mathfrak{d}_{n,\,\alpha ,\,\beta }\he)_{\alpha ,\,\beta \,\in\,J}\he$ in
$\mathscr{C}^{1}\be(X,{N\ee_{X/Y}})$. For any $\alpha $ in $J$, we fix isomorphisms
\[
\xymatrix@C=40pt@R=10pt{
\varphi _{n,\,\alpha }\he\,:\,\bwi{n+1}{\sigma ,\,\lambda _{n}\he}{\oo_{\ba{X\vphantom{X'}}\vert U_{\alpha }\he}}\ar[r]^-{\sim}&\bwi{n+1}{\sigma }{\oo_{\ba{X\vphantom{X'}}\vert U_{\alpha }\he}} \quad \textrm{and} \quad
\psi _{n,\,\alpha }\he\,:\,\bwi{n+1}{\sigma ,\,\mu _{n}\he}{\oo_{\ba{X\vphantom{X'}}\vert U_{\alpha }\he}}\ar[r]^-{\sim}&\bwi{n+1}{\sigma }{\oo_{\ba{X\vphantom{X'}}\vert U_{\alpha }\he}}
}
\]
such that for any $\alpha $, $\beta $ in $J$, $\varphi _{n,\,\beta }\he\circ \varphi _{n,\,\alpha }^{-1}$ and $\psi _{n,\,\beta }\he\circ \psi _{n,\,\alpha }^{-1}$
are given via the isomorphism (\ref{EqDeuxDgAlg}) by
\[
\varphi _{n,\,\beta }\he\circ \varphi _{n,\,\alpha }^{-1}\,(\ubi,\ubj)=(\ubi+\mathfrak{c}_{n,\,\alpha ,\,\beta }\he\we\ubj,\ubj)\quad\textrm{and}\quad
\psi _{n,\,\beta }\he\circ \psi _{n,\,\alpha }^{-1}\,(\ubi,\ubj)=(\ubi+\mathfrak{d}_{n,\,\alpha ,\,\beta }\he\we\ubj,\ubj).
\]
For any integers $i$, $j$ such that $0\le j\le i\le r$, we define inductively cocycles
$(\eta _{i,\,j,\,\uba}\he)_{\uba\,\in\,J^{i-j+1}}\he$ by the formulae
\[
\eta _{i,\,i}\he=1 \quad \textrm{for} \quad 0\le i\le r, \qquad
\eta _{i+1,\,0}\he=(-1)^{i}(\mathfrak{c}_{0}\he-\mathfrak{d}_{i}\he)\we\eta _{i,\,0}\he \quad \textrm{for} \quad 0\le i\le r-1,
\]
\[
\eta _{i+1,\,j}\he=\dfrac{1}{i+1}\,\bigl[ j\,\eta _{i,\,j-1}\he+(-1)^{i-j}\,(\mathfrak{c}_{j}\he-\mathfrak{d}_{i}\he)\we\eta _{i,\,j}\he\bigr]\quad \textrm{for}\quad 1\le j\le i\le r-1.
\]
\par \smallskip
Let $\ti{d}$ be the total differential of the \v{C}ech bicomplex $\mathscr{C}(\pp_{\sigma ,\,\ubm}\he)$. For any integer $k$,
\[
\mathscr{C}(\pp_{\sigma ,\,\ubm}\he)_{k}\he=\bop_{l=\max(0,\,k)}^{r+k}\mathscr{C}^{\,l}\be\bigl( \bwi{l+1-k}{\sigma ,\,\mu _{l-k}\he}{\oo\baxx}\bigr) \quad \textrm{if}\ k\ge -n \quad \textrm{and} \quad
\mathscr{C}(\pp_{\sigma ,\,\ubm}\he)_{k}\he=0 \quad \textrm{if}\ k<-n.
\]
Besides, $\ti{d}=\delta +(-1)^{l}\be\,\wh{d}_{l-k}\he$ on each $\smash[b]{\mathscr{C}^{\,l}\be\bigl( \bwi{l+1-k}{\sigma ,\,\mu _{l-k}\he}{\oo\baxx}\bigr)} $.
For any nonnegative integers $n$ and $k$ such that $n+l\le k$ and any $\alpha _{0}\he,\dots,\,\alpha _{k}^{\vphantom{\bigl)}}$ in $J$, we define two morphisms of sheaves
\[
\xymatrix@C=25pt{
S_{-n,\,l,\,\uba}\he\,:\,\bwi{n+1}{\sigma }{\oo_{\ba{X\vphantom{X'}}\vert U_{\uba }\he}}\ar[r]&\bwi{n+l+1}{\sigma }{\oo_{\ba{X\vphantom{X'}}\vert U_{\uba }\he}} \quad \textrm{and} \quad
T_{-n,\,\lambda ,\,\uba}\,:\,\bwi{n+1}{\sigma ,\,\lambda _{n}\he}{\oo_{\ba{X\vphantom{X'}}\vert U_{\uba }\he}}\ar[r]&\bwi{n+k+1}{\sigma ,\,\mu _{n+l\he}}{\oo_{\ba{X\vphantom{X'}}\vert U_{\uba }\he}}
}
\]
by the formulae
\[
 S_{-n,\,l,\,\uba}\he\,(\ubi,\ubj)=\bigl( (-1)^{l}\be\,\eta _{n+l,\,n,\,\uba}\he\we \ubi,\eta _{n+l,\,n,\,\uba}\he\we \ubj\bigr)\quad\textrm{and}\quad
T_{-n,\,l,\,\uba}\he=\psi ^{-1}_{n+l,\,\alpha _{0}\he}\circ S_{-n,\,l,\,\uba}\he\circ\varphi _{n,\,\alpha _{0}\he}\he.
\]
\par \medskip
By (\ref{EqQuatreDgAlg}), $S_{-n,\,l,\,\uba}\he$ and $T_{-n,\,l,\,\uba}\he$ are $\oo\baxx$-linear. Since the covering $\mathfrak{U}$ is locally finite, the morphisms
$\bigl( T_{-n,\,l,\,\uba}\he\bigr)_{0\le l\le n-r,\ \alpha\,\in\, J^{l+1}\be}\he$ define a morphism $\apl{T_{-n}\he}{\bwi{n+1}{\sigma ,\,\lambda _{n}\he}{\oo\baxx}}{\mathscr{C}\bigl( \pp_{\sigma ,\,\ubm}\he\bigr)_{-n}\he.}$ A tedious but straightforward computation shows that the $\bigl( T_{-n}\he\bigr)_{0\le n\le r}\he$ define an element of $\textrm{Hom}_{\oo\baxx}\he\bigl( \pp_{\sigma ,\ubl}\he,\,\mathscr{C}(\pp_{\sigma ,\,\ubm}\he)\bigr),$ so that we get the following commutative diagram in $D^{\textrm{b}}(\oox)\be$:
\[
\xymatrix@C=40pt@R=30pt{
\pp_{\sigma ,\,\ubl}\he\ar[d]_-{\sim}\ar[r]_-{T}^-{\sim}&\mathscr{C}\bigl( \pp_{\sigma ,\,\ubm}\he\bigr)\ar[d]&\pp_{\sigma ,\,\ubm }\he\ar[d]\ar[l]_-{\sim}\\ \oox\ar[r]^-{\sim}
&\mathscr{C}(\oox)
&\oox\ar[l]_-{\sim}
}
\]
This proves that the quasi-isomorphism ${\apl{\varphi _{\ubm,\,\ubl}\he}{\pp_{\sigma ,\,\ubl}\he}{\pp_{\sigma ,\,\ubm}\he}} $ is obtained by composing the two isomorphisms of the first line. Now we have another commutative diagram, namely
\[
\xymatrix@C=40pt@R=20pt{
\oox\,\lltens{}_{\ooy}\he\,\pp_{\sigma ,\,\ubl}\ar[r]^-{\id\lltensr{}_{\ooy}\he\,T}\ar[d]_-{\sim}&\oox\,\lltensr{}_{\ooy}\he\,\mathscr{C}\bigl( \pp_{\sigma ,\,\ubm}\bigr)\ar[d]&\oox\,\lltensr{}_{\ooy}\he\,\pp_{\sigma ,\,\ubm}\ar[l]_-{\sim}\ar[d]_-{\sim}\\
\oox\,\oti_{\ooy}\he\,\pp_{\sigma ,\,\ubl}\ar[r]^-{\id \oti_{\ooy} \he T}\ar@<-1.4ex>@{}^{\bigl|\!\bigl|}[d]&\oox\,\oti_{\ooy}\he\,\mathscr{C}\bigl( \pp_{\sigma ,\,\ubm}\bigr)\ar@<-1.4ex>@{}^{\bigl|\!\bigl|}[d]&\oox\,\oti_{\ooy}\he\,\pp_{\sigma ,\,\ubm}\ar[l]_-{\sim}\ar@<-1.4ex>@{}^{\bigl|\!\bigl|}[d]\\
\bop_{j=0}^{r}\bw{j}{{N\ee_{X/Y}}[j]}\ar[r]^-{\sim}&\mathscr{C}\bigl( \bop_{i=0}^{r}\bw{i}{{N\ee_{X/Y}}[i]}\bigr)&\bop_{i=0}^{r}\bw{i}{{N\ee_{X/Y}}[i]}\ar[l]_-{\sim}
}
\]
Thus $\Delta _{\sigma }\he(\ubm,\ubl)$ is obtained by composing the two isomorphisms of the last line. The first one is explicitly given by
\[
\xymatrix@C=17pt{\bw{j}{{N\ee_{X/Y}}}\ar[r]&\mathscr{C}^{i-j}\be(\bw{i}{{N\ee_{X/Y}}}),}\quad\xymatrix@C=17pt{
\ubi\ar@{|->}[r]&\eta _{\,i,\,j}\he\we\ubi\!.
}
\]
Hence $\Delta _{\sigma }\he(\ubm,\ubl)_{i, \, j}$ is equal to $\mathfrak{l}_{i,\,j}\he \, (\zeta _{i,\,j}\he)$ where $\zeta _{i,\,j}\he$ is the cohomology class of $\eta _{\,i,\,j}\he$. This finishes the proof.
\end{proof}
\begin{remark} \label{interpretation} The twist of the AK complex by classes in $\textrm{H}^1\bigl(X, N^*_{X/Y}\bigr)$ admits the following geometric interpretation: the existence of a retraction $\sigma$ of $\ba{j}$ implies that the natural sequence
\begin{equation}\label{Pic}
\xymatrix@=20pt{0 \ar[r]&\textrm{H}^{1}(X, N\ee_{X/Y}) \ar[r]&\textrm{Pic}(\,\bax) \ar[r]&\textrm{Pic}(X) \ar[r]& 0}
\end{equation}
is exact. This allows to identify $\textrm{H}^1\bigl(X, N^*_{X/Y}\bigr)$ with isomorphism classes of holomorphic line bundles on $\bax$ whose restriction on $X$ is trivial. Then for any integer $p$ between $0$ and $r-1$ and any class $\mu_p$ in $\textrm{H}^1\bigl(X, N^*_{X/Y}\bigr)$,  if $\mathcal{L}_p$ is a line bundle on $\oo\baxx$ associated with $\mu_p$ and if $\lambda_p=\be \,\mathfrak{l}_{p+1, \,p}(\mu_p)$, it follows from (\ref{EqQuatreDgAlg}) that  $^{\vphantom{\bigl)}}\bwi{p+1}{\sigma ,\,\lambda_{p}}{\oo\baxx}\simeq  \bwi{p+1}{\sigma }{\oo\baxx} \otimes_{\oo\baxx}\he \mathcal{L}\ee_p$.
\end{remark}
We compute $\Delta _{\sigma }\he(\ubm,\ubl)$ in another particular case:
\begin{theorem}\label{ModuleLunaireUn}
 Let $(X, \sigma)$ be a quantized analytic cycle of codimension $r$ in a complex manifold $Y$\!. For any integer $p$ between $0$ and $r-1$, let $\lambda _{p}\he$ and $\mu _{p}\he$ be extension classes in $\emph{Ext}^{1}_{\oox}(\bw{p} N\ee_{X/Y}, \bw{p+1} N\ee_{X/Y})$ such that $\lambda _{p}\he=\mu _{p}\he$ for $p\neq r-1$. Then
\[
\begin{cases}
 \Delta _{\sigma }\he(\ubm,\ubl)_{i,\,i}\he=1&\emph{for}\ 0\le i \le r\\
\Delta _{\sigma }\he(\ubm,\ubl)_{r,\,r-1}\he=\dfrac{1}{r} (\lambda _{r-1}\he-\mu _{r-1}\he)\\
\emph{All other coefficients}\ \Delta _{\sigma }\he(\ubm,\ubl)_{i,\,j}\he\ \emph{vanish}.
\end{cases}
\]
\end{theorem}
\begin{proof}
 We argue exactly as in the proof of Theorem \ref{ThUnCompHKR}. For any integer $n$ between $0$ and $r-1$, we represent the extension classes $\lambda _{n}\he$ and $\mu _{n}\he$ by \v{C}ech cocycles
$\bigl( \mathfrak{c}_{n,\,\alpha ,\,\beta}\he\bigr)_{(\alpha ,\,\beta )\,\in\,J}\he$ and $\bigl( \mathfrak{d}_{n,\,\alpha ,\,\beta}\he\bigr)_{(\alpha ,\,\beta )\,\in\,J}\he$ in
$\smash[t]{\mathscr{C}^{1}\be\bigl( X,\hoo_{\oox}\he(\bw{p}N\ee_{X/Y},\bw{p+1}N\ee_{X/Y})\bigr)} $
so that
\[
\varphi _{n,\,\beta }\he\circ\varphi _{n,\,\alpha }^{-1}(\ubi,\ubj)=(\ubi+\mathfrak{c}_{n,\,\alpha ,\,\beta }\he(\ubj),\ubj)\qquad\textrm{and}\qquad
\psi _{n,\,\beta }\he\circ\psi _{n,\,\alpha }^{-1}(\ubi,\ubj)=(\ubi+\mathfrak{d}_{n,\,\alpha ,\,\beta }\he(\ubj),\ubj).
\]
Then we define the morphisms $S_{-n,\,l,\,\uba}\he$ as follows:
\par \medskip
--\quad$S_{-n,\,0,\,\alpha _{0\he}}\he(\ubi, \ubj)=(\ubi,\ubj)\qquad \textrm{for}\ 0\le n\le r$
\par \smallskip
--\quad$S_{-(r-1),\,1,\,\alpha _{0}\he,\,\alpha _{1}\he}\he(\ubi,\ubj)=\dfrac{1}{r} (\mathfrak{c}_{r-1,\,\alpha _{0}\he,\,\alpha _{1}\he}\he(\ubj)-\mathfrak{d}_{r-1,\,\alpha _{0}\he,\,\alpha _{1}\he}\he(\ubj))$
\par \medskip
--\quad$\textrm{All other}\ S_{-n,\,\lambda ,\,\uba}\he\ \textrm{vanish}$.
\par \medskip
The morphisms $T_{-n,\,l,\,\uba}\he$ define a morphism of complexes from $\pp_{\sigma ,\,\ubl}\he$ to $\mathcal{C}(\pp_{\sigma ,\,\ubm}\he)$, and we conclude as in the proof of Theorem \ref{ThUnCompHKR}.
\end{proof}
As a corollary, we obtain:
\begin{theorem}\label{ModuleLunaireDeux}
 Let $(X, \sigma)$ be a quantized analytic cycle of codimension two in a complex manifold $Y$\!, $\chi $ be an element of $\emph{Der}\he_{\oox}(\oox,N\ee_{X/Y})$, $\wh{\chi }$ be the associated element in
$\emph{Hom}_{\oox}\he(\Omega ^{1}_{X},N\ee_{X/Y})$ and $\emph{at}(N\ee_{X/Y})$ be the Atiyah class of $N\ee_{X/Y}$ in $\emph{Ext}^{1}_{\oox}(N\ee_{X/Y},\Omega ^{1}_{X}\oti N\ee_{X/Y})$.
\par \medskip
If $\theta (\chi )$ denotes the class $\dfrac{1}{2} (\wh{\chi }\we\id)(\emph{at}(N\ee_{X/Y}))$ in
$\emph{Ext}^{1}_{\oox}(N\ee_{X/Y},\bw{2}N\ee_{X/Y})$, the automorphism $\Gamma \he_{\sigma +\chi }\circ\Gamma ^{-1}_{\sigma }$ of $\oox\oplus N\ee_{X/Y}[1]\oplus\bw{2}N\ee_{X/Y}[2^{\vphantom{AA}}]$ is given by the $3 \times 3$ matrix
\[
\begin{pmatrix}
 1&0&0\\0&1&0\\0&\theta (\chi )&1
\end{pmatrix}
\]
\par \medskip
In particular, $\Gamma _{\sigma + \chi}\he\circ\Gamma _{\sigma }^{-1}=\id$ in $\emph{Aut}_{D^{\emph{b}}\be(\C_{X}\he)} (\oox\oplus N\ee_{X/Y}[1]\oplus\bw{2}N\ee_{X/Y}[2]) $.

\end{theorem}
\begin{proof}
 The first part of the theorem follows directly from Theorem \ref{ModuleLunaireUn} and Proposition \ref{PropUnTwisted}. The second part follows from the first one. Indeed, $2\theta (\chi )$ is obtained as the composition
\[
\xymatrix@C=45pt{N\ee_{X/Y}\ar[r]^-{\textrm{at}(N\ee_{X/Y})}&\Omega ^{1}_{X}\oti N\ee_{X/Y}[1]\ar[r]^-{\chi \oti \,\id }&N\ee_{X/Y}\oti  N\ee_{X/Y}[1]\ar[r]^-{\wedge}&\bw{2}N\ee_{X/Y}[1].}
\]
The class $\textrm{at}(N\ee_{X/Y})$ is obtained as the extension class of the exact sequence of $1$-jets of $N\ee_{X/Y}$:
\[
\xymatrix{
0\ar[r]&\Omega ^{1}_{X}\oti N\ee_{X/Y}\ar[r]&J^{1}\be(N\ee_{X/Y})\ar[r]&N\ee_{X/Y}\ar[r]&0.
}
\]
This exact sequence splits over $\C_{X}\he$, so that $\textrm{at}(N\ee_{X/Y})=0$ in $\textrm{Ext}\,^{1}_{\C_{X}\he}\bigl( N\ee_{X/Y},\Omega ^{1}_{X}\oti N\ee_{X/Y}\bigr).$ Thus $\theta (\chi )=0$ in $\textrm{Ext}^{1}_{\C_{X}\he}\bigl( N\ee_{X/Y},\bw{2}N\ee_{X/Y}\bigr).$
\end{proof}
\par \bigskip
We end this section by giving a conjectural expression for the matrix $\Delta _{\sigma }\he(\ubm,\ubl)$. For this purpose we introduce the derived analog of the translation operator defined in \S~\ref{cup}.
\par \bigskip
For any nonnegative integers $m$, $p$ and $k$ such that $k \geq p$ and any $\phi$ in $\textrm{Ext}^{k-p}_{\oox} \bigl(\bw{p}N\ee_{X/Y}, \bw{k}N\ee_{X/Y})$ considered as an element of $\textrm{Hom}^{\vphantom{\bigl)}}_{D^{\textrm{b}}(\oox)} (\bwi{p}{}{{N\ee_{X/Y}}}[p], \,\bwi{k}{}{{N\ee_{X/Y}}}[k])$, we define a morphism $\mathfrak{t}_{k,\,p}^{\, m}(\phi)$ in $\textrm{Hom}^{\vphantom{\bigl)}}_{D^{\textrm{b}}(\oox)}(\bwi{p+m}{}{{N\ee_{X/Y}}}[p+m], \,\bwi{k+m}{}{{N\ee_{X/Y}}}[k+m])$ by the composition
\[
\xymatrix@C=12pt{\st \bwi{p+m}{}{{N\ee_{X/Y}}}[p+m]\ar[rr]^-{\smash[t]{W_{p,\,m}\he} }&&\st\bwi{p}{}{{N\ee_{X/Y}}}[p+m]\,\lltens{}_{\oox}\he\bwi{m}{}N\ee_{X/Y} \ar[rrr]^-{\smash[t]{\phi[m]\, \lltens{}_{\oox}\id} }&&&\st\bwi{k}{}{{ \!\!N\ee_{X/Y}}}[k+m]\, \lltens{}_{\oox}\he \bwi{m}{}\!\!N\ee_{X/Y}\ar[r]^-{\we}&\st\bwi{k+m}{}{{N\ee_{X/Y}}}[k+m]\!
}
\]
The derived version of Lemma \ref{LemUnBisSecAlgExt} tells us that for any class $v$ in $\textrm{H}^{k-p_{\vphantom{)}}}\be\bigl(X, N\ee_{X/Y}\bigr)$,
\[
\mathfrak{t}_{k,\,p} ^{\,m}[\mathfrak{l}_{k, \,p} \he (v)]=\mathfrak{l}_{k+m,\,p+m}(v).
\]
\par \smallskip
This justifies the following conjecture:
\begin{conjecture}\label{conj}
Let $(X, \sigma)$ be a quantized analytic cycle in a complex manifold $Y$\!. For any integer $p$ between $0$ and $r-1$, let  $\lambda _{p}\he$ and $\mu _{p}\he$ be extension classes in ${}^{\vphantom{\bigl)}}\emph{Ext}^{1}_{\oox}\bigl(\bwi{p}{}{{N\ee_{X/Y}}},\bwi{p+1}{}{{N\ee_{X/Y}}}\bigr)$.
\par \bigskip
For any integers $i$ and $j$ such that $0\le j\le i\le r$, we put $\Delta_{i,\,j}=\Delta _{\sigma }\he(\ubm,\ubl)_{i,\,j}\he$.
If $\,*\,$ denotes the Yoneda product, then the coefficients ${}^{\vphantom{iiiiI \bigl(}}\Delta _{i,\,j}\he$ are determined inductively by the following relations:
\par\smallskip
\[
\begin{array}{ll}
-\ \Delta _{i,\,i}\he=\id &\emph{for}\ 0 \leq i \leq r\\[2ex]
-\ \Delta_{i+1,\,0}=(-1)^i \,\,( \mathfrak{l}_{i+1,\,i}(\lambda_0)-\mu_i) * \Delta_{i,\, 0}&\emph {for}\ 0 \leq i \leq r-1\\[1ex]
-\ \Delta _{i+1,\,j}\he=\dfrac{1}{i+1} \bigl[ j\, \mathfrak{t}_{i, \,j-1}^{\,1}(\Delta _{i,\,j-1}\he) \,+(-1)^{i-j}\be\, (\mathfrak{t}_{j+1,\, j}^{\,i} \, \lambda_{j}\he-\mu_{i}\he) * \Delta _{i,\,j}\bigr]&\emph{for}\ 1 \leq j \leq i \leq r-1
\end{array}
\]
\end{conjecture}
\par \medskip
\section{The cycle class of a quantized analytic cycle}\label{cycle}
\subsection{Construction and basic properties of the cycle class}\label{cycle1}
For any complex manifolds $X$ and $Y$ such that $X$ is a closed complex submanifold of $Y$ of codimension $r$, $\rh{\ooy}(\oox,\ooy)$ is canonically isomorphic to $j\ei\omega _{X/Y}\he$ in $D^{\textrm{b}}\be(\ooy)$. This implies that $\rhr{\ooy}(\oox,\ooy)$ is concentrated in degree $r$, so that there exists an isomorphism
\begin{equation}\label{BaProCyClEqUn}
 \rhr{\ooy}(\oox,\ooy)\simeq \omega _{X/Y}\he
\end{equation}
in $D^{\textrm{b}}\be(\oox)$ such that the composition
\[
\xymatrix{j\ei\omega _{X/Y}\he \simeq  j\ei \rhr{\ooy}(\oox,\ooy) \ar[r]^-{\sim} & \rh{\ooy}(\oox,\ooy) \simeq j\ei\omega _{X/Y}\he
}
\] is the multiplication by $(-1)^{\frac{r(r+1)}{2}}$ (the choice of this sign will become clear in the proof of Theorem \ref{BaProCyClThUn} below).
For any integer $i$ between $0$ and $r$, we have an isomorphism
\begin{equation}\label{BaProCyClEqDeux}
 \textrm{Hom}_{D^{\textrm{b}}\be(\oox)}\he(\omega _{X/Y}\he,\bw{i}{N_{X/Y}\he[-i]})\simeq	 \textrm{H}^{r-i}\be(X,\bw{r-i}{N\ee_{X/Y}})
\end{equation}
obtained as follows: for any cohomology class $\alpha $ in $\textrm{H}^{r-i}\be(X,\bw{r-i}{N\ee_{X/Y}})$ considered as a morphism from $\oox$ to $\bw{r-i}{N\ee_{X/Y}[r-i]}$ in $D^{\textrm{b}}\be(\oox)$, we associate the morphism
\[
\xymatrix@C=40pt{
\omega _{X/Y}\he\ar[r]^-{\alpha\, \oti\,\id}&\bw{r-i}{N\ee_{X/Y}[r-i]}\oti\omega _{X/Y}\he\ar[r]^-{\smash{D^{\ell}\be}}&\bw{i}{N_{X/Y}\he[-i].}
}
\]
\begin{definition}\label{BaProCyClDefUn}
Let $(X, \sigma)$ be a quantized analytic cycle of codimension $r$ in a complex manifold $Y$\!. Using the isomorphism (\ref{BaProCyClEqUn}), the morphism
\[
\xymatrix@C=25pt{
\omega _{X/Y}\he\simeq \rhr{\ooy}(\oox,\ooy)\ar[r]&\rhr{\ooy}(\oox,\oox)\ar[r]^-{\sim}_-{\wh{\Gamma }_{\sigma }\he}&\bop_{i=0}^{r}\bw{i}{N_{X/Y}\he[-i]}
}
\]
defines via (\ref{BaProCyClEqDeux}) a class in $\bop_{i=0}^{r}\textrm{H}^{i}\be(X,\bw{i}N\ee_{X/Y})$, which is the \emph{quantized cycle class} $q_{\sigma} \he (X)$ of $(X,\sigma )$.
\end{definition}
We now compute the quantized cycle class in specific situations.
\begin{theorem}\label{BaProCyClThUn}
Let $(X,\sigma )$ be a quantized analytic cycle of codimension $r$ in $Y$ and assume that there exists a couple $(E,s)$ such that
\begin{enumerate}
 \item [(1)] $E$ is a holomorphic vector bundle of rank $r$ on $Y$.
\item[(2)] $s$ is a holomorphic section of $E$ vanishing exactly on $X$ and $s$ is transverse to the zero section.
\item[(3)] The locally-free $\oox$-modules $E\oti_{\ooy}\he\oo\baxx$ and $\sigma \ee\be N_{X/Y}\he$ are isomorphic.\label{CondTrois}
\end{enumerate}
Then $q_{\sigma} \he (X)=1$.
\end{theorem}
\begin{proof}
 Let $s\ee\be$ be the dual of $s$; it is a cosection of $E\ee\be$. Since $s$ is transverse to the zero section, the Koszul complex $(\mathcal{L},\delta )=L(E\ee\be,s\ee\be)$ is a free resolution of $\oox$ over $\ooy$. Using (\ref{EqTroisPrelim}), the iso\-morphism $\rh{\ooy}(\oox,\ooy)\simeq j\ei\omega _{X/Y}\he$ in $D^{\textrm{b}}\be(\ooy)$ is given by the chain
\[
\xymatrix@C=25pt{
\rh{\ooy}(\oox,\ooy)&\LL\ee\be\simeq(\LL,-\delta )\oti_{\ooy}\he\det E[-r]\ar[l]_-{\sim}
\ar[r]^-{\sim}&\oox\oti_{\ooy}\he\det E[-r]\simeq j\ei\omega _{X/Y}\he.
}
\]
If $\tau$ denotes the canonical cosection of $\sigma \ee N\ee_{X/Y}$ and $h$ is the isomorphism between $E\oti_{\ooy}\he\oo\baxx$ and $\sigma \ee\be N_{X/Y}\he$ given in condition (3), then $s\ee\be\! \circ {}^t h$ is a cosection of $\sigma\ee N\ee_{X/Y}$ vanishing on $X$. Hence there exists an endomorphism $F$ of the conormal bundle $N\ee_{X/Y}$ such that  $s\ee\be\! \circ {}^t h$ is obtained as the composition
\[
\xymatrix{\sigma\ee N\ee_{X/Y} \ar[r] & N\ee_{X/Y} \ar[r]^-{F} & N\ee_{X/Y} \ar[r] & \oo \baxx
}
\]
This means that $s\ee\be\! \circ {}^t h=\tau \circ \sigma\ee (F)$. Using again that $s$ is transverse to the zero section, we get that $F$ is an isomorphism. Therefore, if we replace $h$ by $\sigma\ee [\,{}^t F^{-1}] \circ h$, we have $s\ee\be\! \circ {}^t h= \tau$.
\par \medskip
We can now construct a global quasi-isomorphism $\smash{\aplexp{\gamma }{\LL}{\mathcal{P}_{\sigma }\he}{\sim}}$ (which is the global analog of the quasi-isomorphism $\gamma$ constructed in the proof of Proposition \ref{PropUnHKR}) as follows: for $0\le p\le r$, $\gamma _{-p}\he$ is given by the composition
\[
\xymatrix@C=25pt{
\bwi{p}{\ooy}{E\ee\be}\ar[r]&\bwi{p}{\oo\baxx}{(E\ee\be\oti_{\ooy}\he\oo\baxx)}\simeq\bwi{p}{\oo\baxx}{(\sigma \ee\be N\ee_{X/Y})}\simeq\bw{p}{N\ee_{X/Y}\oti_{\oox}\he\oo\baxx}\ar[r]&\bwi{p+1}{\oox}{\oo\baxx}
}
\]
where the last arrow is $\xymatrix@C=17pt{\ub{x}\oti(i,a)\ar@{|->}[r]&(i\wedge\ub{x},a\ub{x}).}$
Let $\smash[t]{\apl{\Delta }{\omega _{X/Y}\he}{\bop_{i=0}^{r}\bw{i}{N_{X/Y}\he}[-i]}}$ be the morphism in $D^{\textrm{b}}\be(\oox)$ defining the quantized cycle class $q_{\sigma} \he (X)$ (c.f. Definition \ref{BaProCyClDefUn}) and let $\psi $ be the automorphism of $\smash[b]{\bop_{i=0}^{r} j\ei\bw{i}{N_{X/Y}\he[-i]}}$ in $D^{\textrm{b}}\be(\ooy)$ such that $(-1)^{\frac{r(r+1)}{2}} \psi$ is given by the composition
\[
\xymatrix@C=25pt{
\bop_{i=0}^{r} j\ei\bw{i}{N_{X/Y}\he[-i]}&j\ei\rhl{\ooy}(\oox,\oox)\simeq j\ei\rhr{\ooy}(\oox,\oox)\ar[l]^-{\Gamma \ee_{\sigma }}_-{\sim}\ar[r]^-{\sim}_-{\wh{\Gamma }_{\sigma }\he}&\bop_{i=0}^{r}j\ei\bw{i}{N_{X/Y}\he[-i].}
}
\]
Then $j\ei\Delta $ can be expressed as the chain
\[
\xymatrix@C=25pt{
j\ei\omega _{X/Y}\he&(\LL,-\delta )\oti_{\ooy}\he\det E[-r]\ar[l]_-{\sim}\ar[r]&\bop_{i=0}^{r}j\ei (\bw{i}{N\ee_{X/Y}[i]}\oti_{\oox}\he\omega _{X/Y}\he)\hspace*{100pt}
}
\]\vspace*{-2ex}
\[
\xymatrix@C=25pt{
\hspace*{200pt}\ar[r]^-{D^{\ell}\be}&\bop_{i=0}^{r}j\ei\bw{i}{N_{X/Y}\he[-i]}\ar[r]^-{\sim}_-{\psi }&\bop_{i=0}^{r}j\ei\bw{i}{N_{X/Y}\he[-i].}
}
\]
Define two morphisms $\ti{\Delta }$ and $\ti{\psi }$ in $D^{\textrm{b}}\be(\oox)$ and $D^{\textrm{b}}\be(\ooy)$ by the diagrams
\[
\xymatrix@C=70pt@R=30pt{
\oox\ar[r]^-{\ti{\Delta }}&\smash{\bop_{i=0}^{r}\bw{i}{N\ee_{X/Y}[i]}}\ar[d]^-{D^{\ell}\be}_-{\sim}\\
\omega _{X/Y}\he\oti_{\ooy}\he\det E\ee\be[r]\ar[u]^-{\sim}\ar[r]^-{\Delta \,\lltens{}_{\ooy}\he\id}&\bop_{i=0}^{r}\bw{i}{N_{X/Y}\he[-i]}\oti_{\ooy}\he\det E\ee\be[r]
}
\]
and
\[
\xymatrix@C=50pt@R=30pt{
\bop_{i=0}^{r}j\ei\bw{i}{N\ee_{X/Y}[i]}\ar[r]^-{\ti{\psi }}_-{\sim}\ar[d]^-{D^{\ell}\be}_-{\sim}&\smash{\bop_{i=0}^{r}j\ei\bw{i}{N\ee_{X/Y}[i]}}\ar[d]^-{\sim}\\
\bop_{i=0}^{r}j\ei\bw{i}{N_{X/Y}\he[-i]}\oti_{\ooy}\he\det E\ee\be[r]\ar[r]^-{\psi  \,\lltens{}_{\ooy}\he\id}&\bop_{i=0}^{r}j\ei\bw{i}{N_{X/Y}\he[-i]}\oti_{\ooy}\he\det E\ee\be[r]
}
\]
\par \bigskip
Then:
\par\medskip
--\quad For $0\le i\le r$, the $i$-th component of $\ti{\Delta }$ in $\textrm{Hom}_{D^{\textrm{b}}\be(\oox)}\he(\oox,\bw{i}{N\ee_{X/Y}}[i])$ is $q_{\sigma} \he (X)_{i}\he\cdot$
\par
--\quad The morphism $j\ei\ti{\Delta }$ is the composition of the chain of morphisms
\[
\xymatrix@C=25pt{
\oox&(\LL,-\delta )\ar[l]_-{\sim}\ar[r]&\bop_{i=0}^{r}j\ei\bw{i}{N\ee_{X/Y}[i]}\ar[r]_-{\sim}^-{\ti{\psi }}&\bop_{i=0}^{r}j\ei\bw{i}N\ee_{X/Y}[i].
}
\]
Using the quasi-isomorphism $\smash{\aplexp{\gamma }{\LL}{\pp_{\sigma \he},}{\sim}}$ we get that $j\ei\ti{\Delta }$ is equal to the composition
\[
\xymatrix@C=25pt
{\oox&(\pp_{\sigma }\he,-\wh{d}\,\,)\ar[l]_-{\sim}\ar[r]&\bop_{i=0}^{r}j\ei\bw{i}{N\ee_{X/Y}[i]}\ar[r]^-{\ti{\psi} }_-{\sim}&\bop
^{r}j\ei\bw{i}{N\ee_{X/Y}[i].}}
\]
We now make the two following observations:
\par\medskip
--\quad As a complex of $\C_{Y}\he$-modules, $\pp_{\sigma }\he$ splits as the direct sum of $\oox $ and a null-homotopic complex.
\par \medskip
--\quad The global version of Proposition \ref{PropDeuxBisHKR} shows that $\ti{\psi }$, as a morphism in the derived category $D^{\textrm{b}}\be(\C_{Y}\he)$, acts by
$(-1)^{\frac{(r-i)(r-i-1)}{2}+r(r-i)+\frac{r(r+1)}{2}}\be$ on each factor $j\ei\bw{i}{N\ee_{X/Y}[i]}$.
\par
Thus, as a morphism in $D^{\textrm{b}}\be(\C_{Y}\he)$, $j\ei \ti{\Delta }$ is simply the injection
$\oox\fl\bop_{i=0}^{r}j\ei \bw{i}{N\ee_{X/Y}[i]}$. Hence we get $q_{\sigma} \he (X)_{0}\he=1$ and
$q_{\sigma} \he (X)_{i}\he=0$ for $1\le i\le r$.
\end{proof}
\par \medskip
As an immediate consequence, we get:
\begin{corollary}\label{BaProCyClCorUn}
For any quantized cycle $(X,\sigma )$, $q_{\sigma} \he (X)_{0}\he=1$.
\end{corollary}
\begin{proof}
 The class $q_{\sigma} \he (X)_{0}\he$ is a holomorphic function on $X$, so that it can be computed locally. Hence we can assume that $X$ is open in $\C^{n}\be$ and $Y=X\tim U$ where $U$ is open in $\C^{r}\be$. If $E$ is the trivial rank $r$ bundle on $Y$ and $s$ is the section $(z_{n+1}\he,\dots, z_{n+r}\he)$, then Theorem \ref{BaProCyClThUn} yields $ q_{\,\pr_{1}} (X)=1$. Since $N_{X/Y}\he$ is trivial, $q_{\sigma} \he (X)$ is independent of $\sigma $ by Proposition \ref{PropDeuxAnalyticHKR} (1). This gives the result.
\end{proof}
We now turn to the case of the diagonal injection. For any complex manifold $X$, we identify the co\-nor\-mal bundle of $\Delta_X\he$ in $X \times X\he$ with $\Omega_{X}^{1}$ as follows: for any germ on holomorphic function $f$ on $X$, the local section $\textrm{pr}\ee_1(f)-\textrm{pr}\ee_2(f)$ of the conormal sheaf of the diagonal corresponds to the local section $df$ of the cotangent bundle of $X$.
\begin{theorem}\label{BaProCyClThDeux}
For any complex manifold $X$, $q_{\,\emph{pr}_1}(\Delta_X \he)$ is the Todd class of $X$.
\end{theorem}
\begin{proof}
If $\qq$ is the dual Atiyah-Kashiwara complex associated with $(\Delta _{X}\he,\pr_{1}\he)$, the main result of \cite{G} is that for a specific isomorphism between $\oox$ and $\rhr{\oo_{X \times X}}(\oox, \omega _{X}\he \boxtimes \oox)$, the composition
\[
\xymatrix@C=25pt{
{\oox\simeq\rhr{\oo_{X \times X}\he}(\oox, \omega _{X}\he \boxtimes \oox)\ar[r]}&{\rhr{\oo_{X \times X}\he}(\oox,\omega _{X}\he)\simeq\hoo_{\oo_{X \times X}\he}\he(\oox,\qq)\simeq\bop_{i=0}^{r}\Omega _{X}^{i}[i]}}
\]
is the Todd class of $X$. It follows that $q_{\,\textrm{pr}_1}(\Delta_X \he)=\varphi \, \textrm{td}(X)$ where $\varphi$ is a nowhere zero holomorphic function on $X$. By Corollary \ref{BaProCyClCorUn}, $\varphi=1$.
\end{proof}
To conclude this section, we compute the quantized cycle class in the case of divisors. For any cohomology class $\delta $ in $H^{1}\be(X,N\ee_{X/Y})$, we denote by $\LL_{\delta }\he$ the associated line bundle on $\bax$. Then, for any $\LL$ in $\textrm{Pic}(\bax)$ such that $j\ee\be\LL\simeq\oox$, there exists a unique cohomology class $\delta $ in $H^{1}\be(X,N\ee_{X/Y})$ such that $\LL$ is isomorphic to $\LL_{\delta }\he$ (c.f. Remark \ref{interpretation}).
\begin{theorem}\label{divisor}
 Let $(X, \sigma)$ be a quantized analytic cycle of codimension one in a complex manifold $Y$\!, and let $\delta $ be the cohomology class in $H^{1}\be(X,N\ee_{X/Y})$ such that $\smash[b]{\ba{j}\ee\be\ooy(X)\oti_{\oo\baxx}\he\sigma \ee\be N\ee_{X/Y}}$ is isomorphic to $\LL_{\delta }\he$. Then $\smash{q_{\sigma} \he (X)=1+\delta.}$
\end{theorem}
\begin{proof}
 Let $\nn$ (resp.\ $\nn'$) denote the holomorphic line bundle $\sigma \ee\be N_{X/Y}\he$ (resp.\ $\ba{j}\ee\be\ooy(X)$) on $\bax$. Then we have two natural exact sequences
\[
\xymatrix@C=17pt{\oo\baxx\ar[r]^-{i}&\nn\ar[r]^-{\pi }&N\he_{X/Y}\ar[r]&0}\qquad\textrm{and}\qquad
\xymatrix@C=17pt{\oo\baxx\ar[r]^-{i'}&\nn'\ar[r]^-{\pi' }&N\he_{X/Y}\ar[r]&0}
\]
If $\apl{\Delta }{N\he_{X/Y}[-1]}{\oox\oplus N\he_{X/Y}[-1]}$ is the morphism in $D^{\textrm{b}}\be(\oox)$ defining $q_{\sigma} \he (X)$, then $\Delta $ is obtained as the composition of quasi-isomorphisms
\[
\xymatrix@C=35pt@R=30pt{
&\oox\ar[r]^-{i}\ar[d]_-{-i'}&\nn\ar[d]^-{(-i'\oti\,\id,-\pi )}&\oox\ar[l]_-{i}\ar[d]\ar@<1.499ex>@{}[d]_(.45){\textrm{\LARGE o}}\\
N_{X/Y}\he&\nn'\ar[l]_-{-\pi '}\ar[r]_-{(\id\oti\,i,\,0)}&[\,\nn'\oti\nn\,]\oplus N_{X/Y}\he\ar[d]^-{(-\pi '\oti\pi ,\,0)}&N_{X/Y}\he\ar[l]^-{(0,\,\id)}\\
&&N_{X/Y}^{\,\oti\,2}
}
\]
Let $\Delta '=\ba{j}\ei\he\Delta\, \lltens{}_{\oo\baxx}\he\,\nn'[1]=\ba{j}\ei\he\bigl( \Delta\, \lltens{}_{\oox}\he N\ee_{X/Y}[1]\bigr).$ If $s$, $s'$, $t$, $t'$ are the maps occurring in the two natural sequences
\[
\xymatrix@C=17pt{0\ar[r]&N\ee_{X/Y}\ar[r]^(0.5){s}&\oo\baxx\ar[r]^-{t}&\oox\ar[r]&0}\qquad\textrm{and}\qquad
\xymatrix@C=17pt{0\ar[r]&N\ee_{X/Y}\ar[r]^(0.5){s'}&\LL_{-\delta }\he\ar[r]^-{t'}&\oox\ar[r]&0}
\]
then $\Delta '$ is the composition
\[
\xymatrix@C=35pt@R=30pt{&N\ee_{X/Y}\ar[r]^-{s'}\ar[d]_-{-s}&\LL_{-\delta }\he\ar[d]_-{(-t',-t')}&N\ee_{X/Y}\ar[l]_(0.5){s'}\ar[d]\ar@<1.48ex>@{}[d]_(.47){\textrm{\LARGE o}}\\
\oox&\oo\baxx\ar[l]^-{-t}\ar[r]_-{(t,\,0)}&\oox\oplus\oox&\oox\ar[l]^-{(0,\,\id)}
}
\]
Thus, as a morphism in $D^{\textrm{b}}\be(\C_{Y}\he)$, $\Delta '$ is the  composition
\[
\xymatrix@C=35pt@R=30pt{&\LL_{-\delta }\he\ar[d]_-{(-t',-t')}&N\ee_{X/Y}\ar[l]_(0.5){s'}\ar[d]\ar@<1.48ex>@{}[d]_(.47){\textrm{\LARGE o}}\\
\oox\ar[r]_(0.5){(-\id,\,0)}&\oox\oplus\oox&\oox\ar[l]^-{(0,\,\id)}
}
\]
Therefore, via the isomorphism
$\textrm{Hom}_{D^{\textrm{b}}\be(\oox)}\he(\oox,\oox\oplus N\ee_{X/Y}[1])\simeq H^{0}\be(X,\oox)\oplus H^{1}\be(X,N\ee_{X/Y})$, we have $\Delta '=1+\delta $. This yields the result.
\end{proof}

\subsection{The six operations for a closed immersion}\label{six}
We denote by $j\ee\be$ (resp.\ $j\pe\be$) the derived pullback (resp.\ exceptional inverse image) induced by the closed immersion $j$. More explicitly,
\begin{align*}
 \apl{j\ee\be}{D^{-}\be(\ooy)}{D^{-}\be(\oox)}&&j\ee\be\ff&=\oox\lltensr{}_{\ooy}\he\ff\\
\apl{j\pe\be}{D^{+}\be(\ooy)}{D^{+}\be(\oox)}&&j\pe\be\ff&=\rhr{\ooy}(\oox,\ff)
\end{align*}
These two functors satisfy the adjunction formulae
\[
\begin{cases}
\textrm{Hom}_{D^{\textrm{b}}\be(\ooy)}\he(\ff,j\ei\he\g)\simeq \textrm{Hom}_{D^{\textrm{b}}\be(\oox)}(j\ee\be\ff,\g)\\
\textrm{Hom}_{D^{\textrm{b}}\be(\ooy)}\he(j\ei\he\g,\ff)\simeq \textrm{Hom}_{D^{\textrm{b}}\be(\oox)}(\g,j\pe\be\ff)
\end{cases}
\]
as well as the projection formula
\[
j\ei \he (j\ee\be\ff \, \lltens{}_{\ooy}\he \, \g)\simeq \ff \, \lltens{}_{\oox}\he \, j\ei \he \g
\]
\par \medskip
for any $\ff$ and $\g$ in $D^{\textrm{b}}\be(\ooy)$ and $D^{\textrm{b}}\be(\oox)$ respectively.
For any element $\ff$ in $D^{\textrm{b}}\be (\oox)$, there is a natural isomorphism
\[
j \ei \he (j \ee \be j \ei \he \ff) \simeq  j \ei \he (j \ee \be j \ei \he \oox \, \lltens{}_{\oox} \he \ff)
\]
\par \medskip
 in ${D^{\textrm{b}}\be(\oox)}\he$ obtained by the chain
\[
j\ei \he (j \ee \be j \ei \he \oox \, \lltens{}_{\oox} \he\ff) \simeq j \ei \he \oox \, \lltens{}_{\ooy} \he \, j \ei \he \ff \simeq  j \ei \he \ff \, \lltens{}_{\ooy}\he \,  j \ei \he \oox \simeq j\ei \he j \ee \be j \ei \he \ff
\]
using the projection formula twice.
\begin{remark}
It is important to notice that for general pairs $(X,Y)$ of complex analytic cycles, the objects $j \ee \be j \ei \he \ff$ and $j \ee \be j \ei \he \oox \, \lltens{}_{\oox} \he \ff$ are not always isomorphic in $D^{\textrm{b}}\be(\oox)$. This can be seen as follows: assuming that $\ff$ is locally free, it is proved in \cite[\S 2.6]{AC} that if $j \ee \be j \ei \he \, \ff$ is formal in $D^{\textrm{b}}\be(\oox)$ then $\ff$ can be lifted to a locally-free sheaf on $\bax$. Therefore, if $N\ee_{X/Y}$ can be lifted to a locally-free sheaf on $\bax$ but $\ff$ cannot, then ${j \ee \be j \ei \he \oox \, \lltens{}_{\oox} \he \ff}$ is formal and $j \ee \be j \ei \he \ff$ is not. Of course, if $j$ admits an infinitesimal retraction $\sigma$, both objects are isomorphic but the isomorphism cannot in general be chosen independent of $\sigma$.
\end{remark}
\par \medskip
For any elements $\ff$ and $\g$ in $D^{\textrm{b}}_{\textrm{coh}}(\ooy)$, the natural morphism
\[
\xymatrix@C=25pt{
\ff \,\lltensl{}_{\ooy}\he\rhr{\ooy}(\oox,\g)\ar[r]&\rhr{\ooy}(\oox,\ff\,\lltens{}_{\ooy}\he\g).
}
\]
\par \smallskip
is an isomorphism. This means that we have an isomorphism
$j\ee\be(\,.\,)\,\lltens{}_{\oox}\he j\pe\be(\,.\,)\fl j\pe\be(\,.\,\,\lltens{}_{\ooy}\he\,.\,)$
of bifunctors from $D^{\textrm{b}}_{\textrm{coh}}(\ooy)\times D^{\textrm{b}}_{\textrm{coh}}(\ooy)$ to $D^{\textrm{b}}_{\textrm{coh}}(\oox)$.
\subsection{Kashiwara's isomorphism}\label{kash}
Let $\hh\hh_{Y}\he(X)$ be the \emph{generalized derived Hochschild complex}. It is defined by
\begin{equation}\label{KashIsoEqUn}
 \hh\hh_{Y}\he(X)=j\ee\be j\ei\he\oox.
\end{equation}
\par \medskip
Then $\hh\hh_{Y}\he(X)$ is a ring object in $D^{\textrm{b}}\be(\oox)$, the multiplication being given by the chain of morphisms
\[
\xymatrix@C=25pt{
j\ee\be j\ei\he\oox\,\lltens{}_{\oox}\he\, j\ee\be j\ei\he\oox\ar[r]^-{\sim}&j\ee\be\bigl( j\ei\he\oox\,\lltens{}_{\ooy}\he \,j\ei\he\oox\bigr)\ar[r]&j\ee\be\bigl( j\ei\he\oox\oti_{\ooy}\he j\ei\he\oox\bigr)=j\ee\be j\ei\he\oox.
}
\]
The object $\bop\limits_{i=0}^{r}\bw{i}N\ee_{X/Y}[i]$ is also a ring object in $D^{\textrm{b}}\be(\oox)$, with multiplication given by cup-product.
\begin{proposition}\label{KashIsoPropUn}
Let $(X, \sigma)$ be a quantized analytic cycle of codimension $r$ in a complex manifold $Y.$ Let $\sigma $ be a retraction of $\ba{j}$. Then
\[
\xymatrix@C=25pt{\Gamma _{\sigma }\he\,:\,\hh\hh_{Y}\he(X)\ar[r]^-{\sim}&\bop_{i=0}^{r}\bw{i}N\ee_{X/Y}[i]}
\]
is a ring isomorphism.
\end{proposition}
\begin{proof}
 We consider the following commutative diagram
\par\medskip
\[
\xymatrix@C=45pt@R=20pt{
\textrm{H}^{0}\be(j\ee\be)(\pp_{\sigma }\he)\oti_{\oox}\he\! \textrm{H}^{0}\be(j\ee\be)(\pp_{\sigma }\he)\ar[r]^-{\sim}&\textrm{H}^{0}\be(j\ee\be)(\pp_{\sigma }\he\oti_{\ooy}\he\!\pp_{\sigma }\he)\ar[r]^-{*}&\textrm{H}^{0}\be(j\ee\be)(\pp_{\sigma }\he)\\
j\ee\be\pp_{\sigma }\he\,\lltens{}_{\oox}\he j\ee\be\pp_{\sigma }\he\ar[u]^-{\sim}\ar[r]^-{\sim}\ar[d]_-{\sim}&j\ee\be\bigl( \pp_{\sigma }\he\,\lltens{}_{\ooy}\he\pp_{\sigma }\he\bigr)\ar[r]\ar[u]\ar[d]_-{\sim}&j\ee\be\pp_{\sigma }\he\ar[u]_-{\sim}\ar[d]^-{\sim}\\
j\ee\be j\ei\he\oox\,\lltens{}_{\oox}\he\, j\ee\be j\ei\he\oox\ar[r]^-{\sim}&j\ee\be \bigl( j\ei\he\oox\,\lltens{}_{\ooy}\he\, j\ei\he\oox\bigr)\ar[r]&j\ee\be j\ei\he\oox
}
\]
\par \bigskip
where $\smash[b]{\apl{*}{\pp_{\sigma }\he\oti_{\ooy}\he\!\pp_{\sigma }\he}{\pp_{\sigma }\he}} $ has been constructed in \S~\ref{DgAlg}. By (\ref{EqTroisDgAlg}), the composition of the two arrows of the first line is the cup-product map via the isomorphism  $\smash[b]{\textrm{H}^{0}\be(j\ee\be)(\pp_{\sigma }\he)\simeq\bop\limits_{i=0}^{r}\bw{i}N\ee_{X/Y}[i]} $. This finishes the proof.
\end{proof}
\begin{remark}\label{KashIsoRemUn}
 This proposition holds in a more general setting, namely when $N\ee_{X/Y}$ extends to $\bax$ and $\Gamma _{\sigma }\he$ is replaced by $\Gamma _{\mathcal{K}}\he$ (where $\mathcal{K}$ is the corresponding untwisted Arinkin-C\u{a}ld\u{a}raru complex). We refer the reader to \cite{AC} for more details.
\end{remark}
The object $j\pe\be j\ei\he\oox$ can also be naturally equipped with an action of $\hh\hh_{Y}\he(X)$. This is done using the chain of morphisms
\begin{equation}\label{KashIsoEqDeux}
j\ee\be j\ei\he\oox\,\lltens{}_{\oox}\he\,j\pe\be j\ei\he\oox \flexp{\sim}j\pe\be\bigl( j\ei\he\oox\,\lltens{}_{\ooy}\he\,j\ei\he\oox\bigr)\fl j\pe\be\bigl( j\ei\he\oox\oti_{\ooy}\he\! j\ei\he\oox\bigr)\flexp{\sim}j\pe\be j\ei\he\oox
\end{equation}
\begin{proposition}\label{KashIsoThUn}
For any quantized analytic cycle $(X, \sigma)$ of codimension $r$ in a complex manifold $Y$\!, the isomorphism
\[
\xymatrix@C=25pt@R=1pt{(\Gamma _{\sigma }\he, \wh{\Gamma} _{\sigma }\he)\,:\,(\hh\hh_{Y}\he(X), \, \,j\pe\be j\ei\he\oox)
\ar[r]^-{\sim}&\bigl(\bop_{i=0}^{r}\bw{i}N\ee_{X/Y}[i], \bop_{i=0}^{r}\bw{i}N\he_{X/Y}[-i]\bigr)
}
\]
preserves the module structure, where $\bop_{i=0}^{r}\bw{i}N\ee_{X/Y}[i]$ acts on $\bop_{i=0}^{r}\bw{i}N\he_{X/Y}[-i]$ by  left contraction.
\end{proposition}
\begin{proof}
Let us consider the following commutative diagram, in which we use implicitly the isomorphism $j\ee\be\pp_{\sigma }\he\flexp{\sim}H^{0}\be(j\ee\be)(\pp_{\sigma }\he)$:
\[
\xymatrix@C=35pt{
\st j\ee\be\pp_{\sigma }\he \,\oti_{\oox}\he\,\textrm{H}^{0}\be(j\pe\be)\,(\,\qq_{\sigma }\he\, \oti_{\oox} \he  \omega_{X/Y} \he)\ar[r]\ar[d]^-{\sim}
&\st\textrm{H}^{0}\be(j\pe\be)\,(\pp_{\sigma }\he\,\oti_{\ooy}\he [\,\qq_{\sigma }\he\, \oti_{\oox} \he  \omega_{X/Y} \he])\ar[r]^-{\id \oti \,\wh{*}\,}\ar[d]
&\st\textrm{H}^{0}\be(j\pe\be)\,(\,\qq_{\sigma }\he\, \oti_{\oox} \he  \omega_{X/Y} \he)\ar[d]^-{\sim}\\
\st j\ee\be\pp_{\sigma }\he\,\lltens{}_{\oox}\,j\pe\be(\,\qq_{\sigma }\he\, \oti_{\oox} \he  \omega_{X/Y} \he) \ar[r]^-{\sim}\ar[d]^-{\sim}
&\st j\pe\be(\pp_{\sigma }\he\,\lltens{}_{\ooy}\he \,[\,\qq_{\sigma }\he\, \oti_{\oox} \he  \omega_{X/Y} \he])\ar[r]\ar[d]^-{\sim}
&\st j\pe\be(\,\qq_{\sigma }\he\, \oti_{\oox} \he  \omega_{X/Y} \he)\\
\st j\ee\be j\ei\he\oox\,\lltens{}_{\oox}\he\, j\pe\be j\ei\he\oox\ar[r]^-{\sim}
&\st j\pe\be(\,j\ei\he \oox \,\lltens{}_{\ooy}\he\, j\ei\he\oox)\ar[r]
&\st j\pe\be j\ei\he\oox\ar[u]_-{\sim}
}
\]
\par \smallskip
where $\wh{*}$ is defined by $(\ref{EqOnzeBis})$. Now we have isomorphisms \[
j\ee\be\pp_{\sigma }\he \simeq \bop_{i=0}^{r}\bw{i}N\ee_{X/Y}[i] \qquad \textrm{and} \qquad \textrm{H}^{0}\be(j\pe\be)\,(\,\qq_{\sigma }\he\,\oti_{\oox}\he\,\omega _{X/Y}\he) \simeq \bop_{i=0}^{r}\bw{i}N_{X/Y}[-i],\]
the second one being given by (\ref{right}). Thanks to (\ref{EqOnzeBis}), a direct computation shows that the composition of the arrows in the first horizontal row of the diagram is via the above isomorphisms the left contraction morphism. This yields the result.
\end{proof}
\begin{definition}\label{KashIsoDefUn}
 For any pair $(X, Y)$ of complex manifolds such that $X$ is a closed complex submanifold of $Y$\!, the \emph{Kashiwara isomorphism} $\Dg$ is a specific isomorphism in the derived category $D^{\textrm{b}}\be(\oox)$ between  $j\ee\be j\ei\he\oox \, \lltens{}_{\oox}\he  \omega _{X/Y}\he$ and $j\pe\be j\ei\he\oox$ given by the chain of morphisms
\[
\xymatrix{{j\ee\be j\ei\he\oox \, \lltens{}_{\oox}\he  \omega _{X/Y}\he \simeq j\ee\be j\ei\he\oox \, \lltens{}_{\oox}\he j\pe\be\ooy \, \ar[r]^-{\sim}}&{\, j\pe\be (j\ei\he\oox \, \lltens{}_{\ooy} \he \ooy })\, \,  \simeq\,\,  j\pe\be j\ei\he\oox .}
\]
\end{definition}
\begin{proposition}\label{KashIsoPropDeux}
 The isomorphism $\Dg$ is an isomorphism of left $\hh\hh_{Y}\he(X)$-modules.
\end{proposition}
\begin{proof}
This follows directly from the commutative diagram
\[
\xymatrix@C=30pt{
\st j\ee\be j\ei\he\oox \, \lltens{}_{\oox}\he\,j\ee\be j\ei\he\oox \, \lltens{}_{\oox}\he j\pe\be\ooy \ar[r]^-{\sim}\ar[d]^-{\sim}_-{\id\, \lltens{}_{\oox}\he \Dg}&\st j\ee\be (j\ei\he\oox\,\lltens{}_{\ooy}\he\, j\ei\oox)\,\lltens{}_{\oox}\he\, j\pe \ooy \ar[r]\ar[d]^-{\sim}&\st j\ee\be j\ei\he\oox \, \lltens{}_{\oox}\he\,j\pe\ooy\,\ar[d]^-{\sim}_-{\Dg}\, \\ \st
j\ee\be j\ei\he\oox\,\lltens{}_{\oox}\he j\pe\be j\ei\he\oox \ar[r]^-{\sim}&\st j\pe\be( j\ei\he\oox\,\lltens{}_{\ooy}\he\, j\ei\he\oox)\ar[r]&\st j\pe\be j\ei\he\oox
}
\]
\end{proof}
We fix an isomorphism between $j\ee\be j\ei\he\oox\,\lltens{}_{\oox}\he\omega _{X/Y}\he$ and $\bop_{i=0}^{r}\bw{i}{N_{X/Y}\he[-i]}$ as follows:
\begin{equation}\label{KashIsoEqTrois}
 \xymatrix@C=50pt{
j\ee\be j\ei\he\oox\,\lltens{}_{\oox}\he\omega _{X/Y}\he\ar[r]^-{\sim}_-{\Gamma _{\sigma }\he\,\lltens{}_{\oox}\he\id}&\bop_{i=0}^{r}\bw{i}{N\ee_{X/Y}[i]}\oti_{\oox}\he\omega _{X/Y}\he\ar[r]^-{\sim}_-{D^{\ell}\be}&\bop_{i=0}^{r}\bw{i}{N_{X/Y}\he[-i].}
}
\end{equation}
Then the main result of this section is:
\begin{theorem}\label{KashIsoThDeux}
 Let $(X,\sigma )$ be a quantized analytic cycle in a complex manifold $Y$\!, and let $M$ be the automorphism of $\bop_{i=0}^{r}\bw{i}{N_{X/Y}\he[-i]}$ occurring in the diagram
\[
\xymatrix@C=40pt@R=20pt{
j\ee\be j\ei\he\oox\,\lltens{}_{\oox}\he\omega _{X/Y}\he\ar[r]^-{\sim}_-{\mathfrak{D}}\ar[d]_-{\sim}&j\pe\be j\ei\he\oox\ar[d]^-{\sim}\\
\bop_{i=0}^{r}\bw{i}{N_{X/Y}\he[-i]}\ar[r]^-{\sim}_-{M}&\bop_{i=0}^{r}\bw{i}{N_{X/Y}\he[-i]}
}
\]
where the left vertical isomorphism is defined by \emph{(\ref{KashIsoEqTrois})}. Then for any integers $i$, $j$ such that $0\le i\le j\le r$, the component $M_{i,\,j}\he$ of $M$ is given by
\[
\xymatrix@C=30pt{
\bw{j}{N_{X/Y}\he[-j]}\ar[rr]^-{\smash[t]{q_{\sigma} \he (X)_{j-i}\he\,\lltens{}_{\oox}\he\id} }&&\bw{j-i}{N\ee_{X/Y}[j-i]}\,\lltens{}_{\oox}\he\bw{j}{N_{X/Y}\he[-j]}\ar[r]^-{\lrcorner}&\bw{i}{N_{X/Y}\he[-i].}
}
\]
In particular, $\mathfrak{D}$ is completely determined by the quantized cycle class $q_{\sigma} \he (X)$.
\end{theorem}
\begin{proof}
 Let $\apl{\Delta }{\omega _{X/Y}\he}{\bop_{i=0}^{k}\bw{i}{N_{X/Y}\he[-i]}}$ be the morphism in $D^{\textrm{b}}\be(\oox)$ defining the quantized cycle class. For any integers $i$ and $j$ such that $0\le i\le j\le r$, Propositions \ref{KashIsoPropUn}, \ref{KashIsoThUn} and \ref{KashIsoPropDeux} imply that $M_{i,\,j}\he$ is given by the composition
\[
\xymatrix@C=35pt{
\bw{j}{N_{X/Y}\he[-j]}&\bw{r-j}{N\ee_{X/Y}[r-j]}\,\lltens{}_{\oox}\he\omega _{X/Y}\he\ar[l]_-{\sim}^-{D^{\ell}\be}\ar[rr]^-{\id\,\lltens{}_{\oox}\he\Delta _{r-j+i}\he}&&\hspace*{150pt}
}
\]
\vspace*{-1ex}
\[
\xymatrix@C=35pt{
\hspace*{100pt}\bw{r-j}{N\ee_{X/Y}[r-j]\,\lltens{}_{\oox}\he}\bw{r-j+i}{N_{X/Y}\he[j-i-r]}\ar[r]^-{\lrcorner}&\bw{i}{N_{X/Y}\he[-i]}
}
\]
which is exactly
\[
\xymatrix@C=30pt{
\bw{j}{N_{X/Y}\he[-j]}&\bw{r-j}{N\ee_{X/Y}[r-j]}\,\lltens{}_{\oox}\he\omega _{X/Y}\he\ar[l]_-{\sim}^-{D^{\ell}\be}\ar[rrr]^-{(\id\we{\, q_{\sigma} \he (X)_{j-i}\he})\,\lltens{}_{\oox}\he\id}&&&\hspace*{100pt}
}
\]
\vspace*{-1ex}
\[
\xymatrix@C=35pt{
\hspace*{150pt}\bw{r-i}{N\ee_{X/Y}[r-i]\,\lltens{}_{\oox}\he}\omega _{X/Y}\he\ar[r]^-{\sim}_-{D^{\ell}\be}&\bw{i}{N_{X/Y}\he[-i].}
}
\]
This yields the result.
\end{proof}

\bibliographystyle{plain}
\bibliography{biblio}

\end{document}